\pdfoutput=1
\RequirePackage{ifpdf}
\ifpdf 
\documentclass[pdftex]{sigma}
\else
\documentclass{sigma}
\fi

\numberwithin{equation}{section}

\newtheorem{Theorem}{Theorem}[section]
\newtheorem{Corollary}[Theorem]{Corollary}
\newtheorem{Lemma}[Theorem]{Lemma}
\newtheorem{Proposition}[Theorem]{Proposition}
 { \theoremstyle{definition}
\newtheorem{Definition}[Theorem]{Definition}
\newtheorem{Example}[Theorem]{Example}
\newtheorem{Remark}[Theorem]{Remark} }
\usepackage{tikz-cd}

\begin{document}

\newcommand{\arXivNumber}{2010.01060}

\renewcommand{\thefootnote}{}

\renewcommand{\PaperNumber}{005}

\FirstPageHeading

\ShortArticleName{Quantum Groups for Restricted SOS Models}

\ArticleName{Quantum Groups for Restricted SOS Models\footnote{This paper is a~contribution to the Special Issue on Representation Theory and Integrable Systems in honor of Vitaly Tarasov on the 60th birthday and Alexander Varchenko on the 70th birthday. The full collection is available at \href{https://www.emis.de/journals/SIGMA/Tarasov-Varchenko.html}{https://www.emis.de/journals/SIGMA/Tarasov-Varchenko.html}}}

\Author{Giovanni FELDER~$^\dag$ and Muze REN~$^\ddag$}

\AuthorNameForHeading{G.~Felder and M.~Ren}

\Address{$^\dag$~Department of Mathematics, ETH Zurich, 8092 Zurich, Switzerland}
\EmailD{\href{mailto:felder@math.ethz.ch}{felder@math.ethz.ch}}

\Address{$^\ddag$~Department of Mathematics, University of Geneva,\\
\hphantom{$^\ddag$}~2-4 rue du Li\`evre, c.p.~64, 1211 Geneva 4, Switzerland}
\EmailD{\href{mailto:muze.ren@unige.ch}{muze.ren@unige.ch}}

\ArticleDates{Received October 05, 2020, in final form January 05, 2021; Published online January 12, 2021}

\Abstract{We introduce the notion of restricted dynamical quantum groups through their category of representations, which are monoidal categories with a forgetful functor to the category of $\pi$-graded vector spaces for a groupoid~$\pi$.}


\Classification{17B37; 18M15}

\begin{flushright}
\begin{minipage}{90mm}
\it Dedicated to Vitaly Tarasov and Alexander Varchenko\\ on their round birthdays
\end{minipage}
\end{flushright}

\renewcommand{\thefootnote}{\arabic{footnote}}
\setcounter{footnote}{0}

\section{Introduction}\label{s-1}
The theory of quantum groups was designed in the 1980s to describe the
algebraic structure underlying the theory of exactly solvable models of
statistical mechanics and quantum mechanics. Since then the theory of quantum groups has entered diverse fields of
mathematics and mathematical physics and the world of exactly solvable
models is entirely explained by quantum groups, in the guise of
Yangians, quantum loop algebras, and elliptic quantum groups. Well,
not entirely~\dots\ One small village of indomitable models, called the
RSOS models still holds out against the invaders.\footnote{Actually there are also other exactly solvable models
 whose quantum group description is unknown, such as the Inozemtsev spin chain \cite{Inozemtsev1990}.}
These Restricted Solid-On-Solid models, introduced in special cases by
Baxter in his studies of the eight-vertex model and the hard hexagon
model, and generalized by Andrews, Baxter and Forrester, are lattice
models of two-dimensional statistical mechanics for which the
technology of exact solutions has provided some of the most
spectacular results. They play a central role also in conformal field
theory (CFT) as their critical behaviour is (or is conjectured to be)
given by the universality classes of minimal unitary CFT models. While
the unrestricted SOS models, whose local degrees of freedom take
values in an infinite set, are by now well-described by the
representation theory of dynamical elliptic quantum groups, the RSOS
models, with finitely many allowed states at every lattice point, are
much less understood.

We propose a theory of dynamical quantum groups with discrete
dynamical parameter with the goal to establish the representation
theory underlying RSOS models and their higher rank generalizations.
We introduce a new approach to this problem, based on groupoid-graded
vector spaces, which may be of independent interest and applicability
in representation theory.
\subsection{Quantum groups and exactly solvable models} The notion of
quantum group \cite{Drinfeld1987} emerged in the Leningrad school in
the 1980s as the algebraic structure underlying exactly solvable
models of statistical mechanics in 2 dimensions and integrable quantum
field theory in $1+1$ dimensions, see~\cite{Faddeev1982}. While
``groups'' may be a misnomer for these Hopf algebras, quantum groups
share with groups the fact that they have an interesting
representation theory for which tensor products of representations are
defined. Excellent textbooks on quantum groups are
\cite{ChariPressleyBook1994, KasselBook1995, LusztigBook1993}. It
soon appeared that quantum groups have a much wider scope of
applications, ranging from low dimensional topology to conformal field
theory, algebraic geometry, gauge theory, representation theory of
affine Lie algebras etc.

Returning to the origin in quantum mechanics and statistical
mechanics, the basic equation is the Yang--Baxter equation, which
appeared in the 1960s in the works of J.B.~McGuire and C.N.~Yang on
one-dimensional many-body quantum systems~\cite{McGuire1964,YangPRL1967} and later in R.~Baxter's work on statistical mechanics~\cite{BaxterAP1972}. In its basic form it is an equation for a
meromorphic function
$z\mapsto R(z)\in\operatorname{End}_{\mathbb C}(V\otimes V)$ of one complex
variable (called the spectral parameter) with values in the linear
endomorphisms of the tensor square of a finite dimensional complex
vector space~$V$. The Yang--Baxter equation is
\[
 R(z-w)^{(12)}R(z)^{(13)}R(w)^{(23)} =
 R(w)^{(23)}R(z)^{(13)}R(z-w)^{(12)}
\]
in $\operatorname{End}(V\otimes V\otimes V)$. Here the superscripts in the
notation indicate the factors on which the endomorphisms act: for
example $R(w)^{(12)}$ means $R(w)\otimes\mathrm{id}$. One also
requires that $R(z)$ is invertible for generic $z$. One of the
simplest non-trivial solutions is the $R$-matrix
$ R(z)=\mathrm{Id}+ z^{-1}P_{VV}$ of McGuire and Yang where $P_{VV}$ is the flip
$v\otimes w\mapsto w\otimes v$. Here $V$ is any finite-dimensional vector space.

As noted by Baxter, solutions of the Yang--Baxter equation give rise
to families of commuting operators on $n$-fold tensor powers
$W=V^{\otimes n}$: fix complex numbers $z_1,\dots, z_n$ and consider
the operator valued function $L(z)\in\operatorname{End}(V\otimes W)$ given
by the product
$R(z-z_n)^{(0n)}\cdots R(z-z_2)^{(02)}
R(z-z_1)^{(01)}$ (we number the factors from $0$ to $n$ where $0$ refers
to the ``auxiliary space''~$V$ in~$V\otimes W$).
Then the (row-to-row) transfer matrices defined by the partial traces
\[
 T(z)=\operatorname{tr}_VL(z)
\]
are commuting endomorphisms of $W=V^{\otimes n}$:
\[
 T(z)T(w)=T(w)T(z),
\]
as a consequence of the Yang--Baxter equation. The relation
to statistical mechanics is that the trace of~$T(z)^m$, say
with all $z_i=0$, when
written out using matrix multiplication, is a sum of products
of matrix entries of~$R(z)$ over all the ways to assigning
a basis vector to pairs of nearest neighbours of an~$n\times m$
lattice with periodic boundary conditions. These assignments
are configurations of local states of a system and the sum is
the partition function of a statistical mechanics model.

The Bethe ansatz, invented by H.~Bethe in 1931 \cite{Bethe1931} in the
case of the Heisenberg spin chain, and further developed by E.~Lieb,
R.~Baxter, B. Sutherland, C.N.~Yang and others in the 1960s, is a
technique to find simultaneous eigenvectors and eigenvalues of the~$T(z)$. The Leningrad school, see \cite{Faddeev1982} for a review of
the results in the early phase, reformulated this technique under the
name ``algebraic Bethe ansatz'' in the framework of the ``quantum inverse scattering
method'' in terms of representation theory of an algebra with
quadratic relations (called $RLL$ or $RTT$
relations) whose coefficients are matrix entries of a solution of the
Yang--Baxter equation. For example the Yangian $Y(\mathfrak{gl}_N)$
corresponds to the McGuire--Yang $R$-matrix with $V=\mathbb C^N$. It can be
defined as the algebra with generators $L_{ij;n}$, $i,j=1,\dots,N$,
$n=1,2,\dots$ with relations
\[
 R(z-w)L(z)\otimes L(w)=L(z)\otimes L(z) R(z-w),\qquad
 R(z)=\mathrm{Id}+\frac 1zP_{VV},
\]
where $L(z)$ is the $N\times N$ matrix with entries
$\delta_{ij}+\sum_{n\geq1}L_{ij,n}z^{-n}$. It is a Hopf algebra
deformation of the universal enveloping algebra of the current Lie
algebra $\mathfrak{gl}_N[t]$ and has a universal $R$-matrix~$\mathcal R$
in a completion of $Y(\mathfrak {gl}_N)\otimes Y(\mathfrak {gl}_N)$
relating the opposite coproduct $\Delta'$ to the coproduct via
$\Delta'(x)=\mathcal R\Delta(x)\mathcal R^{-1}$. Evaluating
$\mathcal R$ in pairs $V_i\otimes V_j$ of finite dimensional
representations of the Yangian yields solutions of the
Yang--Baxter equation, in the generalized form
\begin{equation}\label{e-YB2}
 R_{V_1V_2}^{(12)}R_{V_1V_3}^{(13)}R_{V_2V_3}^{(23)}
 =R_{V_2V_3}^{(23)}R_{V_1V_3}^{(13)}R_{V_1V_2}^{(12)}.
\end{equation}
The spectral parameter may be viewed as a parameter of the
representations $V_i$, and in fact there is an issue of convergence
of the action of $\mathcal R$ on finite dimensional representations,
resulting in the fact that $R_{V_i,V_j}$ is a meromorphic function of
the spectral parameters.

To such a system of $R$-matrices we can associate corresponding
transfer matrices $T_i=\operatorname{tr}_{V_i}R_{V_iV_3}$, $i=1,2$,
acting on $V_3$ and the Yang--Baxter equation with an invertible
$R_{V_1V_2}$ implies that $T_1T_2=T_2T_1$. Baxter's transfer matrices
are the special case $V_{1,2}=V$ and $V_3$ a tensor product of vector
representations with equal spectral parameters.

This story extends to arbitrary semisimple (or reductive) Lie algebras
$\mathfrak g$, and the Yangians~$Y(\mathfrak g)$ provide the algebraic structure
underlying several integrable systems based on rational solutions of
the Yang--Baxter equation such as the Heisenberg spin chain.

The theory admits a trigonometric version, leading to solutions of the
Yang--Baxter equation with trigonometric coefficients. The
corresponding quantum group is a Hopf algebra deformation of the loop
Lie algebra $\mathfrak g\big[t,t^{-1}\big]$. It is (a subquotient of) the
Drinfeld--Jimbo quantum enveloping algebra $U_q\hat {\mathfrak g}$ of the affine
Kac--Moody Lie algebra $\hat {\mathfrak g}$. The corresponding solvable models are
the six-vertex model, a special case of which is the two-dimensional ice model, and
the XXZ spin chain.

\subsection{Elliptic quantum groups and dynamical Yang--Baxter equation}
The next level, after the rational and trigonometric functions, are the
elliptic functions, which are meromorphic functions which are periodic
with respect to two independent periods. Several solvable models have
an elliptic version and the trigonometric and rational versions are
obtained as degenerate limits as the periods tend to infinity.
The relation to quantum groups is more
tricky in the elliptic case. On one hand there is a solution of the
Yang--Baxter equation with elliptic coefficients due to Baxter,
corresponding to the XYZ spin chain and the eight-vertex model, whose
underlying algebraic structure is the Sklyanin algebra (which is not a
Hopf algebra). On the other hand there are the SOS (solid-on-solid) models
also known as IRF (interaction-round-a-face) models. They are
based on a variant of the Yang--Baxter equation, called the
star-triangle relation. While the existence of the
Baxter solution is special for $\mathfrak g=\mathfrak{sl}_N$, we now
know that SOS models exist for all semisimple Lie algebras. Also the
Baxter solution can be related to a solution of the star-triangle relation
by the so-called vertex-IRF transformation, also due to Baxter.

The elliptic quantum groups introduced in \cite{FelderICM1995,
 FelderICMP1995} provide a generalization of the theory of quantum
groups that applies to elliptic SOS models. They are based on a
modification of the Yang--Baxter equation, nowadays called the
dynamical Yang--Baxter equation, see \eqref{e-DYB} below, which had
previously been found by Gervais and Neveu in their study of the
exchange relations of vertex operators in the Liouville conformal
field theory \cite{GervaisNeveu1984}. The dynamical Yang--Baxter
equation reappeared in various contexts since, e.g.,
\cite{AganagicOkounkov2016, CostelloWittenYamazakiI2018,
 CostelloWittenYamazakiII2018,EtingofSchiffmann1999, EtingofVarchenkoCMP1998,JimboKonnoOdakeShiraishi1999,KalmykovSafronov2020,Shibukawa2016,
 StokmanReshetikhin2020}. A recent textbook on elliptic quantum groups
is~\cite{Konno2020}.

The unknown in the dynamical Yang--Baxter equation is a function
$R(z,a)\in \operatorname{End}_{\mathfrak h}(V\otimes V)$ of a second
``dynamical'' variable $a\in\mathfrak h^*$ with values in the dual
vector space to an abelian Lie algebra $\mathfrak h$ and
$V=\oplus_{\mu\in\mathfrak h^*}V_\mu$ is a finite dimensional
semisimple $\mathfrak h$-module.
The (quantum) dynamical Yang--Baxter equation is
\begin{gather}
 R\big(z-w,a+h^{(3)}\big)^{(12)} R(z,a)^{(13)} R\big(w,a+h^{(1)}\big)^{(23)} \notag\\
\qquad =
 R(w,a)^{(23)} R\big(z,a+h^{(2)}\big)^{(13)} R(z - w,a)^{(12)}.\label{e-DYB}
\end{gather}
The ``dynamical shift'' notation is adopted here: for example
$ R(w,a+h^{(3)})^{(12)}$ acts as $R(w,a+\mu_3)\otimes{\mathrm{Id}}$ on
the product of weight subspaces
$V_{\mu_1}\otimes V_{\mu_2}\otimes V_{\mu_3}$. More generally, we can consider
dynamical Yang--Baxter equations in $\operatorname{End}_{\mathfrak h}(V_1\otimes V_2\otimes V_3)$
for $R$-matrices $R_{V_iV_j}(z,a)\in\operatorname{End}_{\mathfrak h}(V_i\otimes V_j)$
as in~\eqref{e-YB2}.

The elliptic quantum group associated with a solution of the dynamical
Yang--Baxter equation and its tensor category of representations can
be again defined by quadratic relations similar to those of the
Yangian but with dynamical shifts at the appropriate places, see
\cite{FelderICM1995,FelderVarchenko1996, GautamToledanoLaredo2017}.
The main new feature is that the representations are vector spaces
over the field of meromorphic functions of the dynamical variables and
the elements of the elliptic quantum group act as difference operators
in these variables. The underlying generalization of the notion of
Hopf algebra was formalized by Etingof and Varchenko
\cite{EtingofVarchenko1998} who called it {\em $\mathfrak h$-Hopf
 algebroid}.

The transfer matrix construction generalizes to the dynamical
setting \cite{FelderVarchenko1996NPB}: suppose that we have invertible operators
$R_{V_iV_j}(z,\lambda)\in\operatorname{End}_{\mathfrak h}(V_i\otimes V_j)$, $i<j\in\{1,2,3\}$
obeying the dynamical Yang--Baxter equation \eqref{e-DYB} on
$V_1\otimes V_2\otimes V_3$, depending meromorphically on $z\in\mathbb C$,
$\lambda\in\mathfrak h^*$. Then the transfer matrix is defined
as an operator acting on meromorphic functions of $\lambda$
with values in the zero-weight subspace of $V_3$:
\begin{equation}\label{e-transfer}
 T_i(z)=\sum_\mu\operatorname{tr}_{V_{i\mu}}R_{V_iV_3}(z,\lambda) t_\mu.
\end{equation}
Here the partial trace is over the weight-$\mu$ subspace of~$V_i$,
the $R$-matrix acts as a multiplication operator and $(t_\mu f)(\lambda)=f(\lambda+\mu)$.
In the special case where $V_1=V_2$, $R_{V_1,V_3}$ has the interpretation of an
$L$-operator obeying a dynamical version of the $RLL$-relations.

\subsection{Restricted SOS models}
The restricted solid-on-solid (RSOS) models introduced by Andrews,
Baxter and Forrester \cite{AndrewsBaxterForrester1984}, generalizing
models previously considered by Baxter \cite{BaxterAP1972,
 BaxterBook1982} in his study of the eight-vertex model and of the
hard hexagon model, form a general class of models of statistical
mechanics in two dimensions. A configuration of a solid-on-solid model
on a subset $M$ of a square lattice in the plane or 2-dimensional torus is
described by assigning an integer $l_i$ (height) to each lattice site
$i\in M$, with the restriction that $|l_i-l_j|=1$ for
neighbouring sites~$i$,~$j$. We can think of the graph of $i\mapsto l_i$
as a discrete random surface modeling the interface between a solid and a gas.

The probability of a configuration $(l_i)$ is proportional to
a product over the faces (unit squares with vertices in $M$)
of Boltzmann weights ${\mathcal W}(l_i,l_j,l_k,l_m)$ depending
on the heights on the corners~$i$, $j$, $k$, $l$ of the face.
In the ``solvable'' SOS models the Boltzmann weights are part of
a one-parameter family ${\mathcal W}(z;a,b,c,d)$ obeying the star-triangle
relation
\begin{gather*}
\sum_g{\mathcal W}(z-w;f,g,d,e)
{\mathcal W}(z;a,b,g,f)
{\mathcal W}(w;b,c,d,g)
 \\
\qquad{} = \sum_g
 {\mathcal W}(w;a,g,e,f)
 {\mathcal W}(z;c,d,e,g)
 {\mathcal W}(z-w;a,b,c,g),
\end{gather*}
which is best understood graphically:
\[
\begin{tikzpicture}[scale=1]
 \draw (0,0) node[left]{$\displaystyle{\sum_g}\quad a$}
 -- (.5,-.866);
 \draw (0,0) -- (.5, .866);
 \draw (.5, .866) node[above]{$f$} -- (1,0);
 \draw (.5,-.866) node[below]{$b$}-- (1,0);
 \draw (.5, .866) -- (1.5, .866);
 \draw (.5,-.866) -- (1.5, -.866);
 \draw (1,0) node[left]{$g$}-- (2,0);
 \draw (1.5, .866) node[above]{$e$} -- (2,0);
 \draw (1.5, -.866) node[below]{$c$} -- (2,0) node[right]{$d$};
\end{tikzpicture}
\begin{tikzpicture}[scale=1]
 \draw (0,0) node[left]{$ =\quad\displaystyle{\sum_g}\quad a$}
 -- (.5,-.866);
 \draw (0,0) -- (.5, .866);
 \draw (1,0) node[right]{$g$} -- (1.5, .866);
 \draw (1.5,-.866) node[below]{$c$} -- (1,0);
 \draw (.5, .866) node[above]{$f$} -- (1.5, .866);
 \draw (.5,-.866) node[below]{$b$} -- (1.5, -.866);
 \draw (0,0) -- (1,0);
 \draw (1.5, .866) node[above]{$e$} -- (2,0);
 \draw (1.5, -.866) node[below]{$c$} -- (2,0) node[right]{$d$};
\end{tikzpicture}
\]
The star-triangle equation admits interesting families of solutions in
terms of elliptic theta functions. Andrews, Baxter and Forrester
considered a special limit of parameters so that the equation holds
for the heights in a finite interval
\[
 l_i\in\{1,2,\dots,r-1\}.
\]
For these families of models (depending essentially on an elliptic
curve and a point of order $r$ on it) they were able to compute
several quantities in the thermodynamic limit $M\to\mathbb Z^2$
(under some physically motivated assumptions on the asymptotic
behaviour), including the probability distribution of the height at
the origin as a function of the boundary conditions in the ordered
phase. One interesting mathematical outcome of this calculation is
that it involves for $r=5$ (Baxter's hard hexagon model) the
celebrated Rogers--Ramanujan identities, which get generalized to
arbitrary $r$. From the point of view of statistical mechanics and
conformal field theory, these models are interesting since their
scaling limit at the critical point are conjectured \cite{Huse1984} to
be the unitary $A$-series of minimal models of
Belavin--Polyakov--Zamolodchikov and Friedan--Qiu--Shenker.

We will be mostly concerned with a generalization of the RSOS models
in which the heights take values in the weight lattice of a simple Lie
algebra, see \cite{DateJimboMiwaOkado1986, JimboKunibaMiwaOkado1988,
 JimboMiwaOkado1988}. The main difference is that in general the
Boltzmann weights ${\mathcal W}(a,b,c,d)$ are no longer scalar-valued, but must
be understood as linear operators.

The relation with the dynamical quantum groups comes from
the simple observation that
the above star-triangle equation is essentially a rewriting of the
dynamical Yang--Baxter equation. The row-to-row transfer matrix
of the RSOS model is the transfer matrix \eqref{e-transfer} for
suitable representations of the elliptic quantum group associated
with $\mathfrak{gl}_2$, acting on functions with support on a finite set.

This restriction of a difference operator such as \eqref{e-transfer}
with meromorphic coefficients to a~finite set is rather subtle as one
needs to avoid the poles and check that the support condition is
preserved. This was done in the case of the RSOS model in~\cite{FelderVarchenko1999}, where it was also shown that the
$\mathfrak{gl}_2$-elliptic weight functions of
\cite{FelderTarasovVarchenkoAMS1997} obey ``resonance conditions'',
guaranteeing that their restrictions to suitable discrete or finite
subsets of the values of the dynamical variable provide, via the Bethe
ansatz, well-defined eigenvectors of the row-to-row transfer matrix of
the RSOS model.

\subsection{Categories of representations}
Instead of talking of quantum groups it is more convenient to talk
about their tensor category of representations, and we take this
approach in this paper. In the representation theory of elliptic
quantum groups \cite{EtingofLatour2005, FelderVarchenko1996}, the
representation space of a representation is defined as a~graded vector
space over the field of meromorphic functions of the dynamical
variables, where the grading is by weights of the underlying Lie
algebra. The representation structure is defined by {\em
 $\mathbb C$-linear} endomorphisms obeying quadratic relations and
commutations relations with scalar multiplication by meromorphic
functions.

For the application to RSOS models, where the dynamical variables take
values in a discrete set, the approach with meromorphic functions is
not suitable. In this paper we propose that the vector spaces
underlying representations of quantum groups with discrete dynamical
variables should be groupoid-graded vector spaces. More precisely we
propose that representations of such quantum groups are monoidal
categories equipped with a faithful monoidal functor to the category
of $\pi$-graded vector spaces of finite type for a certain groupoid~$\pi$ (see Section~\ref{s-2} for the definitions). For applications to
generalized RSOS models the groupoids are certain subgroupoids of the
transformation groupoid for the translation action of the weight
lattice of a semisimple Lie algebra. It turns out that in this
approach the various shifts of dynamical variables appearing in the
dynamical context appear naturally and one can immediately apply the
standard technology of the quantum inverse scattering method
($R$-matrices, $RLL$ relations, transfer matrices, Bethe ansatz). An
instance of this is the fusion procedure, which consists in
constructing solutions of the Yang--Baxter or star-triangle relations
from known ones by taking subquotients of tensor products.

An interesting new feature in the groupoid-graded case is that the
Grothendieck ring of the category of $\pi$-graded vector spaces is
non-commutative in general, even in the case of action groupoids of
abelian groups. Thus characters of representations of dynamical
quantum groups live in a non-commutative ring. However if a collection
of representations $(V_i)$ admit $R$-matrices, which are isomorphisms
$V_i\otimes V_j\cong V_j\otimes V_i$, then their characters generate
a commutative subring of the Grothendieck ring of $\pi$-graded vector
spaces. In the case of transformation groupoids these rings are
realized as rings of commuting difference operators.

\subsection{Outline of the paper}
We introduce the category $\mathrm{Vect}_k(\pi)$ of $\pi$-graded
vector spaces of finite type over a field $k$ in Section \ref{s-2}.
It is a variant of a special case of the category of $\pi$-graded modules
considered in~\cite{Lundstroem2004}. It is an abelian monoidal
category with duality. We discuss the notion of character of a
$\pi$-graded vector space taking values in the convolution ring of
$\pi$. In Section \ref{s-3} we adapt the machinery of Yang--Baxter
equations and transfer matrices to the case of $\pi$-graded vector spaces
and explain the relation with the star-triangle relation. We introduce the
notion of partial traces in this context and prove that solutions of
the Yang--Baxter equation give rise to commuting transfer matrices. In
the case of transformation groupoids and their subgroupoids, we show
that the Yang--Baxter equation can be written as a dynamical
Yang--Baxter equation, and that transfer matrices produce commuting
difference operators. In Section~\ref{s-4} we consider in more detail
the example of the elliptic quantum group of type $A_{n-1}$, which
admits a dynamical $R$-matrix with restricted dynamical variables and
thus a monoidal category with a forgetful functor to $\pi$-graded
vector spaces for a finite groupoid $\pi$. We compute a few
characters, in particular the characters of (analogues of the)
exterior powers of the vector representation, obtained by the fusion
procedure. Finally in Section~\ref{s-5} we consider the case of
dynamical $R$-matrices arising from quantum groups at root of unity,
which may be viewed as a toy model for restricted models, with
$R$-matrices that are independent of the spectral parameters. The construction
uses a semisimple rigid braided category~$C_q(\mathfrak g)$ of representations
of quantum groups for each simple
Lie algebra~$\mathfrak g$ and root of unity~$q$.
Technically it is a semisimple quotient of the
category of tilting modules of the Lusztig quantum groups.
It has finitely many isomorphism classes of simple objects.
We construct a~faithful monoidal functor from $C_q(\mathfrak g)$ to the
categories of $\pi$-graded vector spaces of finite type for a suitable
finite groupoid $\pi$. The braiding in $C_q(\mathfrak g)$ is then mapped
to a system of dynamical $R$ matrices with dynamical variable restricted to
a finite set. This construction is a formalization of the ``passage to
the shadow world'' of \cite{KirillovReshetikhin1988,Turaev1994} and is a
version with discrete dynamical variable of \cite{EtingofVarchenko1999}.
The characters of simple modules define a representation of the Verlinde
algebra by difference operators.

\section{Grading by groupoids}\label{s-2}
\subsection{Groupoids}
A groupoid $\pi$ on a set $A$ is a small category with set of objects $A$
whose morphisms, called arrows, are invertible. The set of morphisms
from an object $a$ to an object~$b$ is denoted by~$\pi(a,b)$. The
composition of arrows $\gamma\in\pi(a,b)$ and $\eta\in\pi(b,c)$ is
denoted by $\eta\circ\gamma\in\pi(a,c)$ or by $\eta\gamma$ in case of
typographical constraints. The inverse of $\gamma\in\pi(a,b)$ is
$\gamma^{-1}\in\pi(b,a)$. We identify $A$ with the subset of identity
arrows and denote a groupoid by its set of arrows $\pi$ when no
confusion arises. The maps $s,t\colon\pi\to A$ sending
$\gamma\in\pi(a,b)$ to $a$ and $b$, respectively, are called source
and target map, respectively.

A subgroupoid of a groupoid $\pi$ is a subset of (the set of arrows
of) $\pi$ that is closed under composition and inversion. It is a
groupoid on the set of its identity arrows. The {\em full
 subgroupoid} of $\pi$ on a subset $B\subset A$ is
the subgroupoid $s^{-1}B\cap t^{-1}B$ of arrows between objects of $B$.

\subsection{The convolution ring of a groupoid}
To a groupoid $\pi$ we associate the {\em convolution ring of $\pi$},
which is a unital associative ring $\mathbb Z(\pi)$ with an involutive anti-automorphism.

As an abelian group $\mathbb Z(\pi)$
consists of the maps $n\colon \pi\to \mathbb Z$ such that for all
$a\in A$, the set of arrows $\alpha\in s^{-1}(a)\cup t^{-1}(a)$ with
$n(\alpha)\neq 0$ is finite. The product is the convolution product
\[
 n*m\colon \ \gamma\mapsto \sum_{\beta\circ\alpha=\gamma}
 n(\alpha)m(\beta).
\]
The sum has finitely many non-zero terms because of the finiteness
assumption. The unit is the characteristic function on identity arrows
and the involutive anti-automorphism $\sigma$ sends $n$ to
$\sigma(n)\colon\gamma\to n\big(\gamma^{-1}\big)$. The assignment
$\pi\mapsto \mathbb Z(\pi)$ is a contravariant functor from the
category of groupoids to the category of unital involutive associative
rings.
\begin{Remark}\label{r-1} For any commutative ring $R$ we have an $R$-algebra
 $R(\pi)=R\otimes_{\mathbb{Z}}\mathbb Z(\pi)$ obtained by extension of
 scalars. For transfer matrices we will need a more general
 construction where~$R$ is also $\pi$-graded, see Section~\ref{ss-cac} below.
\end{Remark}
\subsection{Convolution rings of subgroupoids}
It will be convenient to view the convolution ring of a subgroupoid
$\pi'\subset \pi$ as a subring of $\mathbb Z(\pi)$.
\begin{Lemma}
 The characteristic functions $\chi_{A'}$ of subsets $A'\subset A$ of
 the set of identity arrows are idempotents in $\mathbb Z(\pi)$.
\end{Lemma}
\begin{proof}
 By definition $\chi_{A'}(\gamma)=0$ unless $\gamma$ is an identity
 arrow $a\in\pi(a,a)$ for $a\in A'$. In
 this case $\chi_{A'}(a)=1$. Thus $\chi_{A'}*\chi_{A'}(\gamma)$
 vanishes unless $\gamma$ is an identity arrow $a\in A'$, in which
 case $\chi_{A'}*\chi_{A'}(a)=\chi_{A'}(a)\chi_{A'}(a)=1$.
\end{proof}
\begin{Lemma}\label{l-subgr} Let $\pi$ be a groupoid on $A$ and $\pi'$
 the full subgroupoid on $A'\subset A$. Then the
 induced morphism $\mathbb Z(\pi)\to \mathbb Z(\pi')$ restricts
 to a unital ring isomorphism
 \[
 \chi_{A'}*\mathbb Z(\pi)*\chi_{A'}\to \mathbb Z(\pi'),
 \]
 where $\chi_{A'}$ is the unit element of the subring
 $\chi_{A'}*\mathbb Z(\pi)*\chi_{A'}$. Moreover the left-hand side
 is the subring of functions vanishing on the complement of $\pi'$.
\end{Lemma}
\begin{proof}
 The map $\mathbb Z(\pi)\to \mathbb Z(\pi')$ is the restriction map
 $r\colon n\mapsto n|_{\pi'}$. The extension by zero
 $\mathbb Z(\pi')\to\mathbb Z(\pi)$ is a right inverse. Its image
 consists of the functions vanishing outside $\pi'$. Thus $r$
 restricts to an isomorphism from the functions vanishing outside
 $\pi'$ and $\mathbb Z(\pi)$. Now a function $n\in\mathbb Z(\pi)$
 vanishes outside the full subgroupoid $\pi'$ if and only it
 vanishes everywhere except on arrows between elements of $A'$. But this is
 equivalent to $n=\chi_{A'}*n*\chi_{A'}$. Since $\chi_{A'}$ is
 an idempotent, $\chi_{A'}*\mathbb Z(\pi)*\chi_{A'}$ is a subring
 with unit element $\chi_{A'}$, which is sent to $1\in\mathbb Z(\pi')$.
\end{proof}
\subsection{Action groupoids} The main examples of groupoids for our
purpose are action groupoids and their subgroupoids. Let $G$ be a
group with identity element $e$ and $A$ be a set with a right action
$A\times G\to A$. The action groupoid $A\rtimes G$ has set of objects $A$ and
an arrow $a\to a'$ for each $g\in G$ such that $a'=ag$. Thus an arrow
is described by a pair $(a,g)\in A\times G$. The source and target are
$s(a,g)=a$, $t(a,g)=ag$ and the composition is
\[
 (a',g')\circ(a,g)=(a,gg'),\qquad \text{whenever} \quad a'=ag.
\]
The identity arrows are $(a,e)$, $a\in A$ and the inverse of $(a,g)$
is $\big(ag,g^{-1}\big)$.

The convolution ring $\mathbb Z(A\rtimes G)$ contains the subring
$\mathbb Z^A$ of functions with support on the identity arrows
as in the general case and a subring $\mathbb ZG$ isomorphic
to the group ring of $G$ via the injective ring homomorphism
$t\colon\mathbb ZG\to \mathbb Z(A\rtimes G)$ sending $g\in G$ to
\[
 t_g\colon \ (a,h)\mapsto \delta_{g,h}.
\]
The right action of $G$ defines a group homomorphism
$r\colon G\to \operatorname{Aut}(\mathbb Z^A)$: for $g\in G$ and
$f\in\mathbb Z^A$, $r_gf(a)=f(ag)$.
\begin{Proposition}\label{p-difference}
 The convolution ring $\mathbb Z(A\rtimes G)$ is the crossed
 product $\mathbb Z^A\rtimes_{r} \mathbb ZG$ of its sub\-rings~$\mathbb Z^A$ and $\mathbb ZG$. The involution acts trivially
 on $\mathbb Z^A$ and as $t_g\mapsto t_{g^{-1}}$ on $\mathbb ZG$.
\end{Proposition}
This means that $\mathbb Z(A\rtimes G)$ is isomorphic to the algebra
generated by $\mathbb Z^A$ and elements $t_g$ for $g\in G$ with
relations
\[
 t_gt_h=t_{gh},\qquad t_gf=r_g(f)t_g,\qquad g,h\in G, \quad f\in \mathbb Z^A.
\]
Explicitly, a function $n\in\mathbb Z(A\rtimes G)$ corresponds to the element
$\sum_{g\in G}n_gt_g$ where $n_g(a)=n(a,g)$.
\begin{Remark} In particular the convolution ring acts on the space of
 functions on $A$ by difference operators (i.e., operators acting on
 functions by translations of the argument and multiplication by
 functions). This is the scalar case of the more general case of a
 convolution algebra acting on vector-valued functions on $A$ by
 difference operators, which we construct in Section~\ref{ss-3.7}.
\end{Remark}
\subsection[The category of pi-graded vector spaces of finite type]{The category of $\boldsymbol{\pi}$-graded vector spaces of finite type}
\begin{Definition}
 Let $\pi$ be a groupoid with set of objects $A$. A {\em $\pi$-graded vector
 space of finite type} over a field $k$ is a collection
 $(V_\alpha)_{\alpha\in\pi}$ of finite-dimensional vector spaces
 indexed by the arrows of $\pi$ such that for each $a\in A$ there are
 finitely many arrows $\alpha$ with source or target~$a$ and nonzero $V_\alpha$.
\end{Definition}
The $\pi$-graded vector spaces over $k$ form an abelian
category $\operatorname{Vect}_k(\pi)$: the $k$-vector space
$\operatorname{Hom}(V,W)$ of morphisms between objects $V,W$
consists of families $(f_\alpha)_{\alpha\in\pi}$
of linear maps $f_\alpha\colon V_\alpha\to W_\alpha$ and the composition
is defined componentwise.

\subsection{Tensor product}\label{ss-tp}
The finite type condition allows us to define a monoidal structure
(tensor product) on $\operatorname{Vect}_k(\pi)$. The tensor product
of objects is
\begin{equation}\label{e-tensorproduct}
 (V\otimes W)_\gamma=\oplus_{\beta\circ\alpha=\gamma}
 V_\alpha\otimes W_\beta.
\end{equation}
The direct sum is over all pairs of arrows whose composition is
$\gamma$ and has finitely many nonzero summands. Similarly the tensor
product $f\otimes g$ of morphisms has components
$\oplus_{\beta\circ\alpha=\gamma}f_\alpha\otimes g_\beta$. For any
three objects $U$, $V$, $W$ of $\operatorname{Vect}_k(\pi)$ and $\delta\in\pi$,
\begin{gather*}
 ((U\otimes V)\otimes W)_\delta
 =\oplus_{\alpha\circ\beta\circ\gamma=\delta}
 (U_\gamma\otimes V_{\beta})\otimes W_\alpha, \\
 (U\otimes (V\otimes W))_\delta
 =\oplus_{\alpha\circ\beta\circ\gamma=\delta}
 U_\gamma\otimes (V_{\beta}\otimes W_\alpha).
\end{gather*}
Therefore the associativity constraint in $\operatorname{Vect}_k$
defines an associativity constraint
\[
 \alpha_{UVW}\colon (U\otimes V)\otimes W\to U\otimes(V\otimes W)
\]
in $\operatorname{Vect}_k(\pi)$. The tensor unit in
$\operatorname{Vect}_k(\pi)$ is
$\mathbf 1=(\mathbf 1_\gamma)_{\gamma\in\pi}$ with $\mathbf 1_a=k$ for
identity arrows $a\in A$ and $\mathbf 1_\gamma=0$ for all other
arrows. Then for every object $V$ of $\operatorname{Vect}_k(\pi)$,
$(\mathbf 1\otimes V)_\gamma=k\otimes V_\gamma$ and
$(V\otimes \mathbf 1)_\gamma=V_\gamma\otimes k$. Thus the structure
isomorphisms $V_\gamma\cong V_\gamma\otimes k\cong k\otimes V_\gamma$
of $\operatorname{Vect}_k$ define natural isomorphisms
$\lambda\colon V\cong V\otimes \mathbf 1$,
$\rho\colon V\cong \mathbf 1\otimes V$ in
$\operatorname{Vect}_k(\pi)$.
\subsection{Duality}\label{ss-duality}
The dual of a $\pi$-graded vector space
$V\in\operatorname{Vect}_k(\pi)$ is the $\pi$-graded vector space
$V^*$ with components
$(V^*)_\gamma= \operatorname{Hom}_k(V_{\gamma^{-1}},k)$. Recall that an
object $V$ of a monoidal category admits a left dual of an
object $V$ if there is an object
$V^\vee$, called left dual of $V$, together with morphisms
$\delta\colon \mathbf 1 \to V\otimes V^\vee$,
$\operatorname{ev} \colon V^\vee\otimes V\to \mathbf 1 $ such that the
compositions
\begin{gather*}
 V\cong \mathbf 1\otimes V\to \big(V\otimes V^\vee\big)\otimes V\cong
 V\otimes\big(V^\vee\otimes V\big)\to V\otimes \mathbf 1\cong V, \\
 V^\vee\cong V^\vee\otimes \mathbf 1 \to V^\vee\otimes \big(V\otimes V^\vee\big)\cong
 \big(V^\vee\otimes V\big)\otimes V^\vee\to \mathbf 1 \otimes V^\vee
\end{gather*}
are equal to the identity morphism. Similarly one has the notion
of right dual object ${}^\vee V$ with morphisms $\mathbf 1 \to {}^\vee V\otimes V$,
$V\otimes {}^\vee V\to \mathbf 1 $. Right and left duals of finite dimensional
vector spaces coincide.
\begin{Lemma}\label{l-pivotal}
 Let $V\in \mathrm{Vect}_k(\pi)$. Then $V^\vee$ with components
 $\big(V^\vee\big)_\gamma=(V_{\gamma^{-1}})^*=\operatorname{Hom}_k(V_{\gamma^{-1}},k)$,
 and structure morphisms induced by those of the category of finite
 dimensional vector spaces is both left and right dual to $V$.
\end{Lemma}
For example the morphism $\delta\colon \mathbf 1\to V\otimes V^\vee$
is the collection of maps
$\mathbf 1 _a=k\to \oplus_{\gamma\in s^{-1}(a)} V_\gamma\otimes
\big(V^\vee\big)_{\gamma^{-1}}$, defined by the canonical map
$k\to V_\gamma\otimes (V_{\gamma})^*$.

Monoidal categories admitting left and right duals for all objects,
which are then uniquely determined up to unique isomorphism, are called rigid. Left
and right dualities are monoidal functors to the opposite categories
with opposite tensor product. Rigid monoidal categories with coinciding
left and right dual functors are called pivotal, see \cite[Sections~1.6 and~1.7]{TuraevVirelizier2017}
for more details.
\begin{Theorem} The $k$-additive category $\operatorname{Vect}_k(\pi)$
 with the tensor product $\otimes$, the tensor unit $\mathbf 1$,
 associativity constraint $\alpha$, left and right multiplication by
 the tensor unit~$\lambda$, $\rho$, and duality $( \ )^\vee$
 is an abelian pivotal monoidal category.
\end{Theorem}
This is an immediate consequence of the fact that
$\operatorname{Vect}_k$ is a $k$-additive abelian monoidal category
and Lemma \ref{l-pivotal}.
\begin{Remark}
 Contrary to the case of finite dimensional vector spaces, the
 mo\-noid\-al category $\operatorname{Vect}_k(\pi)$ is not symmetric
 or braided, so that $V\otimes W$ is not isomorphic to $W\otimes V$
 in general. As we will see presently, the Grothendieck ring is not
 commutative in general.
\end{Remark}
\begin{Remark}\label{r-2} The above construction works for any $k$-additive
 rigid monoidal category $C$ over a commutative ring $k$ instead of
 $\operatorname{Vect}_k$. The resulting category of $\pi$-graded
 objects of $C$ of finite type is a $k$-additive monoidal category.
 For example, if we view a ring as a monoidal category with one
 object, the convolution ring~$\mathbb Z(\pi)$ is the category of
 $\pi$-graded objects of finite type of~$\mathbb Z$. One can also
 replace $\pi$ by a general small category, at the cost of giving up duality.
\end{Remark}

\subsection{Characters} The {\em character}
$\operatorname{ch}_V\in\mathbb Z(\pi)$ of
$V\in\operatorname{Vect}_k(\pi)$ is the map
$\gamma\mapsto \mathrm{dim}(V_\gamma)$.
\begin{Lemma} Let $V,W\in \operatorname{Vect}_k(\pi)$. Then
 \begin{gather*}
 \operatorname{ch}_{\mathbf 1} =1,\qquad \operatorname{ch}_{V^\vee}=\sigma(\operatorname{ch}_V), \\
 \operatorname{ch}_{V\oplus W}= \operatorname{ch}_V+\operatorname{ch}_W, \\
 \operatorname{ch}_{V\otimes W}=\operatorname{ch}_V*\operatorname{ch}_W
 \end{gather*}
\end{Lemma}
Since exact sequences of vector spaces split, the
character map $\operatorname{ch}\colon V\mapsto \operatorname{ch}_V$
descends to a~ring homomorphism from the Grothendieck ring
$K(\operatorname{Vect}_k(\pi))$ to the convolution ring
$\mathbb Z(\pi)$.
\begin{Proposition} The map $K(\operatorname{Vect}_k(\pi))\to\mathbb Z(\pi)$
 is an isomorphism of involutive unital rings.
\end{Proposition}
The inverse map sends $n=n_+-n_-$ with $n_\pm(\gamma)\geq0$ for all $\gamma$
to the formal difference $\big[(k^{n_+(\gamma)})_{\gamma\in\pi}\big]-\big[(k^{n_-(\gamma)})_{\gamma\in\pi}\big]$.

\subsection{Subgroupoids}\label{ss-subg}
Let $i\colon\pi'\hookrightarrow \pi$ be a subgroupoid. Then we have an
exact fully faithful functor
$i_*\colon\operatorname{Vect}_k(\pi')\to\operatorname{Vect}_k(\pi)$ so that
\[
 (i_*V)_\gamma=\begin{cases}V_{\gamma},&\text{if $\gamma\in\pi'$,}\\
 0,&\text{otherwise.}
 \end{cases}
\]
We can thus view $\mathrm{Vect}_k(\pi')$ as a full subcategory of
$\mathrm{Vect}_k(\pi)$ of $\pi$-graded vector spaces $V$ such that
$V_\gamma=0$ for $\gamma\not\in\pi'$.

\subsection[Convolution algebras with coefficients in pi-graded algebras]{Convolution algebras with coefficients in $\boldsymbol{\pi}$-graded algebras}\label{ss-cac}

In the setting of $\pi$-graded vector spaces the natural home for
transfer matrices is convolution algebras with coefficients in
$\pi$-graded algebras over a field (or commutative ring)~$k$.
\begin{Definition} Let $\pi$ be a groupoid. A $\pi$-graded algebra $R$
 over $k$ is a collection $(R_\gamma)_{\gamma\in\pi}$ of $k$-vector
 spaces labeled by arrows of $\pi$ with bilinear products
 $ R_\alpha\times R_\beta\to R_{\beta\circ\alpha}$,
 $(x,y)\mapsto xy$, defined for composable arrows $\alpha$, $\beta$
 and units $1_a\in R_a$, for $a\in A$ such that
 (i) $(xy)z=x(yz)$
 whenever defined and (ii) $x1_b=x=1_ax$ for all $x\in R_\alpha$ of degree
 $\alpha\in\pi(a,b)$.
\end{Definition}
\begin{Remark}
 Am algebra object in $\operatorname{Vect}_\pi$ defines a
 $\pi$-graded algebra, but we will need to consider more general
 examples which do not necessarily fulfill the finite type condition,
 such as $\operatorname{\underline{End}}\mathbf 1$ below.
\end{Remark}
\begin{Example}\label{e-1}
 Let $V\in\operatorname{Vect}_k(\pi)$ and let
 $\operatorname{\underline{End}}V$ be the $\pi$-graded vector space
 with
 $(\operatorname{\underline{End}}V)_\alpha=
 \oplus_{\gamma\in\pi(a,a)}
 \operatorname{Hom}_k(V_{\alpha\circ\gamma\circ\alpha^{-1}},V_\gamma)$,
 where $a=s(\alpha)$. Then $\operatorname{\underline{End}}V$ with the
 product given by the composition of linear maps
 \[
 \operatorname{Hom}_k(V_{\alpha\gamma\alpha^{-1}},V_\gamma) \otimes
 \operatorname{Hom}_k (V_{\beta\alpha\gamma\alpha^{-1}\beta^{-1}},
 V_{\alpha\gamma\alpha^{-1}}) \to \operatorname{Hom}_k
 (V_{\beta\alpha\gamma(\beta\alpha)^{-1}},V_\gamma)
 \]
 and unit $1_a=\oplus_{\gamma\in\pi(a,a)}\mathrm{Id}_{V_\gamma}$ is a
 $\pi$-graded algebra.
\end{Example}
\begin{Definition}
 Let $R$ be a $\pi$-graded algebra. The {\em convolution algebra
 $\Gamma(\pi,R)$ with coefficients in $R$} is the $k$-algebra of maps
 $f\colon\pi\to \sqcup_{\alpha\in\pi}R_\alpha$ such that
 \begin{enumerate}\itemsep=0pt
 \item[(i)] $f(\alpha)\in R_\alpha$ for all arrows $\alpha\in\pi$,
 \item[(ii)] for every $a\in A$, there are finitely many $\alpha\in s^{-1}(a)\cup t^{-1}(a)$
 such that $f(\alpha)\neq0$.
 \end{enumerate}
 The product is the convolution product
 \[
 f*g(\gamma)=\sum_{\beta\circ\alpha=\gamma}f(\alpha)g(\beta).
 \]
\end{Definition}
\begin{Example}\label{e-2} Let $R=\operatorname{\underline{End}}\mathbf 1$
 be the $\pi$-graded algebra of Example \ref{e-1} for the tensor unit
 $\mathbf 1$. Then $R_\alpha=k$ for all arrows $\alpha$ and
 $\Gamma(\pi,R)=k(\pi)=k\otimes \mathbb Z(\pi)$ is the extension of
 scalars of the convolution ring of $\pi$, see Remark \ref{r-1}.
\end{Example}

\begin{Lemma} Let $\pi$ be a groupoid with object set $A$.
 The convolution algebra $\Gamma(\pi,R)$ with coefficients in a $\pi$-graded
 algebra $R$ is an associative unital $k$-algebra. The unit is the
 map $a\mapsto 1_a$ for identity arrows $a\in A$ and $\alpha\mapsto 0$
 for other arrows.
\end{Lemma}

Let $\varphi\colon R\to R'$ be a morphism of $\pi$-graded algebras.
Then $\varphi_*\colon \Gamma(\pi,R)\to \Gamma(\pi,R')$ given
$\varphi_*(f)=\varphi\circ f$ is an algebra homomorphism. This
defines a functor from the category of $\pi$-graded $k$-algebras to
associative unital algebras. In particular we have a morphism of
algebras $k(\pi)\to \Gamma(\pi,R)$ for any~$R$.

\section[Yang--Baxter equation and RLL relations]{Yang--Baxter equation and $\boldsymbol{RLL}$ relations}\label{s-3}
\subsection{Yang--Baxter equation}

Let $k=\mathbb C$.
A {\em Yang--Baxter operator} on
$V\in \operatorname{Vect}_k(\pi)$ is a meromorphic function
$z\mapsto\check R(z)\in\operatorname{End}(V\otimes V)$ of the
{\em spectral parameter} $z\in\mathbb C$ with values in the endomorphisms
of $V\otimes V$, obeying the Yang--Baxter equation\footnote{The Yang--Baxter equation
 is usually formulated for the operator $R(z)=p\circ\check R(z)$ obtained
 by composition with the flip $p\colon u\otimes v\mapsto v\otimes u$. Since~$p$ is not
 a morphism of $\pi$-graded vector spaces in general, it is better to use~$\check R(z)$.}
\begin{equation}\label{e-YB}
 \check R(z-w)^{(23)}\check R(z)^{(12)}
 \check R(w)^{(23)}
 = \check R(w)^{(12)}
\check R(z)^{(23)}
\check R(z-w)^{(12)}
\end{equation}
in $\operatorname{End}(V\otimes V\otimes V)$ for all generic values of the spectral
parameters $z$, $w$, and the inversion (also called unitarity) relation
\begin{equation*}
 \check R(z)\check R(-z)=\mathrm{id}_{V\otimes V},
\end{equation*}
for generic $z$.
The restriction of $\check R(z)$ to $V_\alpha\otimes V_\beta$ for
composable arrows $\alpha$, $\beta$ has components in each direct
summand of the decomposition~\eqref{e-tensorproduct}:
\[
 \check R(z)|_{V_\alpha\otimes V_\beta}=\oplus_{\gamma,\delta}
 {\mathcal W}(z;\alpha,\beta,\gamma,\delta).
\]
The sum is over $\gamma$, $\delta$ such that
$\beta\circ\alpha=\delta\circ\gamma$, and $\mathcal W$ is the component
\[
 {\mathcal W}(z;\alpha,\beta,\gamma,\delta)\in\operatorname{Hom}_k
 (V_\alpha\otimes V_\beta, V_\gamma\otimes V_\delta).
\]
The Yang--Baxter equation translates to its IRF
(interaction-round-a-face) version, called the star-triangle relation
\begin{gather*}
 \sum_{\rho,\sigma,\tau}
{\mathcal W}(z-w;\rho,\sigma,\epsilon,\delta)^{(23)}
 {\mathcal W}(z;\alpha,\tau,\zeta,\rho)^{(12)}
 {\mathcal W}(w;\beta,\gamma,\tau,\sigma)^{(23)}
 \\
\qquad{} =
 \sum_{\rho,\sigma,\tau}
 {\mathcal W}(w;\sigma,\tau,\zeta,\epsilon)^{(12)}
 {\mathcal W}(z;\rho,\gamma,\tau,\delta)^{(23)}
 {\mathcal W}(z-w;\alpha,\beta,\sigma,\rho)^{(12)}.
\end{gather*}
There is one such equation for all
$\alpha,\dots,\zeta$ so that
$\gamma\circ\beta\circ\alpha=\delta\circ\epsilon\circ\zeta$ and the
sum is over arrows $\rho$, $\sigma$, $\tau$ for which all factors are
defined, namely such that the diagrams are commutative in~$\pi$.

It is convenient to have a graphical representation for these
morphisms:\tikzset{>=latex}
\[
\begin{tikzpicture}[scale=1.5]
 \draw[->,dashed](-.6,.5) node[left]{ ${\mathcal W}(z;\alpha,\beta,\gamma,\delta)=$}
 -- (1.7,.5);
 \draw[->,dashed](.5,-.6) 
 -- (.5,1.6);
 \draw[->](0,1) node[left]{$a$} -- (.5,1) node[above,fill=white]{$\gamma$} -- (1,1)
 node[right]{$d$};
 \draw[->](1,1) -- (1,.5) node[right , fill=white]{$\delta$} -- (1,0);
 \draw[->](0,1) -- (0,.5) node[left, fill=white]{$\alpha$} -- (0,0);
 \draw[->]node[left]{$b$}(0,0) -- (.5,0) node[below,fill=white]{$\beta$}-- (1,0)
 node[right]{$c$};
 \end{tikzpicture}
\]
The dashed lines are associated with the vector spaces between which
$\mathcal W$ acts: moving from the southwest to the northeast
according to the orientation of the dashed lines we move from
$V_\alpha\otimes V_\beta$ to $V_\gamma\otimes V_\delta$. We have also
displayed in the corners the objects between which the arrows
$\alpha,\dots,\delta$ are defined. For example $\alpha$ is an arrow
from~$a$ to~$b$. In the literature one often considers the case where
there is at most one arrow from one object to any other object, as is the case
in the original Andrews--Baxter--Forrester RSOS models, and it
is then customary to label $\mathcal W$ by the four objects $a$, $b$, $c$, $d$
instead of the morphisms.

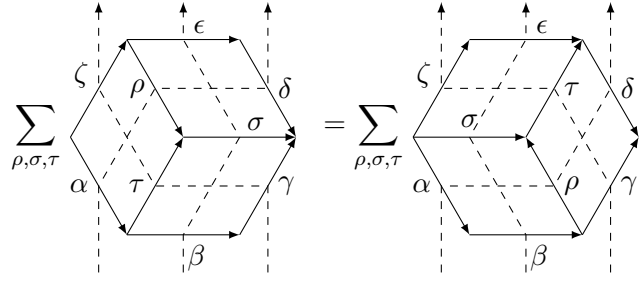
\begin{figure}\centering
 \begin{tikzpicture}[scale=1.5]
 \draw[->] (0,0) node[left]{$\displaystyle{\sum_{\rho,\sigma,\tau}}$}
 -- (.5,-.866);
 \draw[->] (0,0) -- (.5, .866);
 \draw[->] (.5, .866) -- (1,0);
 \draw[->] (.5,-.866) -- (1,0);
 \draw[->] (.5, .866) -- (1.5, .866);
 \draw[->] (.5,-.866) -- (1.5, -.866);
 \draw[->] (1,0) -- (2,0);
 \draw[->] (1.5, .866) -- (2,0);
 \draw[->] (1.5, -.866) -- (2,0);
 \draw[->,dashed]
 (.25, -1.2) 
 -- (.25, -.433) node[left]{$\alpha$}
 -- (.75, .433) node[left]{$\rho$}
 -- (1.75, .433) node[right]{$\delta$}
 -- (1.75, 1.2);
 \draw[->,dashed]
 (1, -1.2) 
 -- (1, -.866) node[right=5pt,below]{$\beta$}
 -- (1.5, 0) node[above=5pt,right=-1pt]{$\sigma$}
 -- (1, .866) node[above=5pt,right]{$\epsilon$}
 -- (1, 1.2);
 \draw[->,dashed]
 (1.75, -1.2) 
 -- (1.75, -.433) node[right]{$\gamma$}
 -- ( .75, -.433) node[left]{$\tau$}
 -- ( .25, .433) node[above=5pt,left]{$\zeta$}
 -- ( .25, 1.2);
\end{tikzpicture}
\begin{tikzpicture}[scale=1.5]
 \draw[->] (0,0) node[left]{$\displaystyle{=\sum_{\rho,\sigma,\tau}}$}
 -- (.5,-.866);
 \draw[->] (0,0) -- (.5, .866);
 \draw[->] (1,0) -- (1.5, .866);
 \draw[->] (1.5,-.866) -- (1,0);
 \draw[->] (.5, .866) -- (1.5, .866);
 \draw[->] (.5,-.866) -- (1.5, -.866);
 \draw[->] (0,0) -- (1,0);
 \draw[->] (1.5, .866) -- (2,0);
 \draw[->] (1.5, -.866) -- (2,0);
 \draw[->,dashed]
 (.25, -1.2) 
 -- (.25, -.433) node[left]{$\alpha$}
 -- (1.25, -.433) node[right]{$\rho$}
 -- (1.75, .433) node[right]{$\delta$}
 -- (1.75, 1.2);
 \draw[->,dashed]
 (1, -1.2) 
 -- (1, -.866) node[right=5pt,below]{$\beta$}
 -- (.5, 0) node[above]{$\sigma$}
 -- (1, .866) node[above=5pt,right]{$\epsilon$}
 -- (1, 1.2);
 \draw[->,dashed]
 (1.75, -1.2) 
 -- (1.75, -.433) node[right]{$\gamma$}
 -- (1.25, .433) node[right]{$\tau$}
 -- ( .25, .433) node[above=5pt,left]{$\zeta$}
 -- ( .25, 1.2);
\end{tikzpicture}
\caption{The Yang--Baxter equation. The arrows $\alpha,\dots,\zeta$ form a commutative hexagon and the sum is over arrows $\rho$, $\sigma$, $\tau$ making the squares commutative.}
\end{figure}

\subsection[RLL relations]{$\boldsymbol{RLL}$ relations}\label{ss-RLL}

The machinery of the quantum inverse scattering method~\cite{Faddeev1982} can be applied: given a solution
$\check R(z)\in \operatorname{End}(V\otimes V)$ of the dynamical Yang--Baxter equation
for a $\pi$-graded vector space $V$, an {\em $L$-operator}
on $W\in\operatorname{Vect}_k$ is a meromorphic function $z\mapsto
L(z)\in\operatorname{Hom}(V\otimes W,W\otimes V)$ such that
\begin{enumerate}\itemsep=0pt
\item[(i)] $L(z)$ is invertible for generic $z$,
\item[(ii)]
 $L$ obeys the {\it RLL} relations
\[
 \check R(z-w)^{(23)}
 L(z)^{(12)}L(w)^{(23)}
=L(w)^{(12)}L(z)^{(23)}
 \check R(z-w)^{(12)},
\]
in $\operatorname{Hom}(V\otimes V\otimes W,W\otimes V\otimes V)$.
\end{enumerate}
For example $\check R(z)$ is an $L$-operator on $V$ thanks to the Yang--Baxter equation.

Given a basis of $V$ the {\it RLL} relations may be written as relations for
the matrix entries $L_{ij}(z)\in\operatorname{End}(W)$. Thus
$L$-operators may be understood as $\pi$-graded meromorphic
representations of the quadratic algebra $A_R$ with generators
$L_{ij}(z)$, and {\it RLL} relations. Here meromorphic refers to the
required meromorphic dependence on $z\in\mathbb C$.

\subsection[The monoidal category M(R,pi)]{The monoidal category $\boldsymbol{M(R,\pi)}$}
The $L$-operators form an abelian mo\-noid\-al category $M(R,\pi)$ of
$\pi$-graded meromorphic representations of~$A_R$: an object
$(W,L_W)$ is a $\pi$-graded vector space
$W\in\operatorname{Vect}_k(\pi)$ endowed with an $L$-operator $L_W$
on~$W$. A morphism from $(W,L_W)$ to $(Z,L_Z)$ is a morphism
$f\colon W\to Z$ of $\pi$-graded vector spaces such that
\[
 (f\otimes\mathrm{id}_V)L_W(z)=L_Z(z)(\mathrm{id}_V\otimes f),
\]
for all $z$.

The tensor product $(W\otimes Z,L_{W\otimes Z})$
is the tensor product
in $\operatorname{Vect}_k(\pi)$ endowed with the composition
\[
 L_{W\otimes Z}(z)\colon \
 V\otimes W\otimes Z
 \xrightarrow{L_W(z)\otimes \mathrm{id}_Z}
 W\otimes V\otimes Z
 \xrightarrow{\mathrm{id}_W\otimes L_Z(z)}
 W\otimes Z\otimes V.
\]
The fact that $L_{W\otimes Z}$ is an $L$-operator is a straightforward
consequence of the definitions. We have an action of $\mathbb C$ on
the category $M(R,\pi)$: for each $u\in \mathbb C$ let $t_u$ be the
endofunctor sending an object $(W,L_W)$ to $(W,L_W(\cdot+u))$ and a
morphism~$f$ to~$f$. Clearly $t_0$ is the identity endofunctor and
$t_ut_v=t_{u+v}$. Moreover $t_u$ is a monoidal functor: the obvious
map $t_u(W)\otimes t_u(Z)\to t_u(W\otimes Z)$ is a natural
isomorphism.
\begin{Example}\label{e-0}
 Let $V\in M(R,\pi)$ be the representation with $L$-operator
 $\check R$. Then for each $u\in\mathbb C$, the representation
 $V(u)=t_uV$ has $L$ operator $L_{V(u)}(z)=\check R(z+u)$. This
 object of $M(R,\pi)$ is called the vector representation with evaluation
 point $u$.
\end{Example}
\begin{Example}
 Let $\mathbf 1\in \operatorname{Vect}_k(\pi)$ be the tensor unit,
 see Section \ref{ss-tp} and let $L_{\mathbf 1}(z)$ be the
 composition
 $\rho\lambda^{-1}\colon V\otimes \mathbf 1\to V\to \mathbf 1\otimes
 V$ of the structure isomorphisms. Then $\mathbf 1$ with this
 $L$-operator is a representation, called the trivial
 representation. It is fixed by the action of $t_u$.
\end{Example}
\begin{Example}
 The dual representation of a representation $(W,L_W)$ admitting a
 dual is the representation $\big(W^\vee,L_{W^\vee}\big)$ on the $\pi$-graded
 dual vector space $W$ (see Section~\ref{ss-duality}). Its $L$-operator
 $L_{W^\vee}(z)=\tilde L_W(z)^{-1}$ is the inverse of the dual
 operator $\tilde L_W(z)\colon W^\vee\otimes V\to V\otimes W^\vee$
 defined as the composition
 \begin{align*}
 W^\vee\otimes V
 \to W^\vee\otimes V\otimes W\otimes W^\vee
 \xrightarrow{L_W(z)^{(23)}} W^\vee\otimes W\otimes V\otimes W^\vee
 \to V\otimes W^{\vee}
 \end{align*}
 with the structure maps defining the duality in the category of
 $\pi$-graded vector spaces, see Section \ref{ss-duality}. It exists whenever
 $\tilde L_W(z)$ is invertible for generic $z$.
\end{Example}
\subsection[R-matrices]{$\boldsymbol{R}$-matrices}
Let $W,Z\in M(R,\pi)$. An isomorphism
\[
 \check R_{W,Z}\colon \ W\otimes Z\to Z\otimes W
\]
in $M(R,\pi)$ is called an $R$-matrix. This means that
$\check R_{W,Z}$ is a morphism of $\pi$-graded vector spaces obeying
the intertwining relation
\[
 \check R_{W,Z}^{(12)}L_Z(z)^{(23)}L_W(z)^{(12)}=
 L_W(z)^{(23)}L_Z(z)^{(12)}\check R_{W,Z}^{(23)}
\]
in $\operatorname{Hom}(V\otimes W\otimes Z,Z\otimes W\otimes V)$.
\begin{Example}
Let $V(u)$ be the vector representation with evaluation point~$u$.
Then the Yang--Baxter equation implies that
 $\check R(u-v)$ is a morphism
 $V(u)\otimes V(v)\to V(v)\otimes V(u)$. It is an $R$-matrix except if
 $u-v$ or $v-u$ is a pole of $\check R$.
\end{Example}

\begin{Proposition}\quad
 \begin{enumerate}\itemsep=0pt
 \item[$(i)$] If $\check R_{W,Z}$ is an $R$-matrix for $W,Z\in M(R,\pi)$
 and $u\in\mathbb C$, then the same isomorphism $\check R_{W,Z}$ of
 $\pi$-graded vector spaces is an $R$-matrix for $t_uV$, $t_uW$.
 \item[$(ii)$] If $\check R_{W,Z}$, $\check R_{W,Z'}$ are $R$-matrices
 then $\check R_{W,Z\otimes Z'}=\check R_{W,Z'}^{(23)}\check R_{W,Z}^{(12)}$
 is an $R$-matrix for $W$, $Z\otimes Z'$.
 \item[$(iii)$] If $\check R_{W,Z}$, $\check R_{W',Z}$ are $R$-matrices
 then
 $\check R_{W\otimes W',Z}\!=\!\check R_{W,Z}^{(12)}\check R_{W',Z}^{(23)}$ is an $R$-matrix for $W\otimes W'$,~$Z$.
 \end{enumerate}
\end{Proposition}

\subsection{Partial traces and transfer matrices}

The partial trace over $V$ is the map
\[
 \operatorname{tr}_V\colon \
 \operatorname{Hom}_{\operatorname{Vect}_k(\pi)}(V\otimes W,W\otimes
 V) \to \Gamma(\pi,\operatorname{\underline{End}} W)
\]
defined as follows.

For $f\in \operatorname{Hom}(V\otimes W,W\otimes V)$
and $\alpha\in\pi(a,b),\gamma\in\pi(a,a)$, let $f(\alpha,\gamma)$
be the component of $f$ mapping
\[
 f(\alpha,\gamma)\colon \ V_\alpha\otimes W_{\alpha\gamma\alpha^{-1}}
 \to W_{\gamma}\otimes V_\alpha.
\]
Define
\[
 \operatorname{tr}_{V_\alpha}f (\alpha,\gamma)=\sum_i(\mathrm{id}\otimes e_i^*)
 f(\alpha,\gamma) (e_i\otimes
 \mathrm{id})\in\operatorname{Hom}(W_{\alpha\gamma\alpha^{-1}},W_\gamma).
\]
for any basis $e_i$ of $V_\alpha$ and dual basis $e_i^*$ of the dual
vector space $(V_\alpha)^*$.
\begin{Definition}
 The {\em partial trace} $\operatorname{tr}_V f\in
 \Gamma(\pi,\operatorname{\underline{End}}W)$ of
 $f\in \operatorname{Hom}_{\operatorname{Vect}_\pi} (V\otimes
 W,W\otimes V)$ over $V$ is the section
 \[
 \operatorname{tr}_V f\colon \ \alpha\mapsto\oplus_{\gamma\in\pi(a,a)}
 \operatorname{tr}_{V_\alpha}f(\alpha,\gamma)
 \in(\operatorname{\underline{End}}W)_\alpha.
 \]
\end{Definition}
\begin{Example}\label{e-4} Let $W=\mathbf 1$ be the tensor unit, with
 nonzero components $W_a=k$, indexed by identity arrows $a\in A$.
 For $\alpha\in\pi(a,b)$, we have
 $(\operatorname{\underline{End}}
 W)_\alpha=\operatorname{Hom}_k(W_a,W_b)=k$, see Example~\ref{e-2}.

The convolution algebra
$\Gamma(\pi,\operatorname{\underline{End}}\mathbf 1)$ is the
extension of scalars of the convolution ring of~$\pi$.
The partial trace of the identity $V\otimes \mathbf 1\cong V\to V\cong \mathbf 1\otimes V$ is
\[
 \operatorname{tr}_V(\mathrm{id})\colon \ \alpha\mapsto \operatorname{dim}V_\alpha,
\]
which is the (image in $k$ of the) character ch${}_V$ of $V$.
\end{Example}

\begin{Lemma}\label{l-JB}\hfill
 \begin{enumerate}\itemsep=0pt
 \item[(i)] If $\varphi\colon V\to V'$ is an isomorphism of
 $\pi$-graded vector spaces then
 \[
 \operatorname{tr}_{V'}\big((\mathrm{id}\otimes\varphi)
 f\big(\varphi^{-1}\otimes\mathrm{id}\big)\big)=\operatorname{tr}_{V}f.
 \]
 \item[(ii)] Let $f_i\in\operatorname{Hom}(V_i\otimes W,W\otimes V_i)$,
 $i=1,2$, be homomorphisms of $\pi$-graded vector spaces and let
 $f_1^{(12)}f_2^{(23)}$ be the composition
 \[
 V_1\otimes V_2\otimes W\xrightarrow{\mathrm{id}\otimes
 f_2}V_1\otimes W\otimes
 V_2\xrightarrow{f_1\otimes\mathrm{id}}W\otimes V_1\otimes V_2.
 \]
 Then
 \[
 \operatorname{tr}_{V_1\otimes V_2}f_1^{(12)}f_2^{(23)}
 =\operatorname{tr}_{V_1}f_1 \operatorname{tr}_{V_2}f_2.
 \]
 \end{enumerate}
\end{Lemma}
\begin{proof}
Recall that a morphism $\varphi$ is a collection of linear maps
$\varphi_\alpha\colon V_\alpha\to V'_\alpha$, so (i) is the
standard property of the trace on each $V_\alpha$.

As for (ii) to compute $\operatorname{tr}_{V_1\otimes V_2}$ we need to
select for each $\gamma\in\pi(a,b)$ and $\mu\in\pi(b,b)$ the component
$g(\alpha)$ of $g=f_1^{(12)}f_2^{(23)}$ sending
$(V_1\otimes V_2)_\gamma\otimes W_\mu$ to
$W_{\mu'}\otimes (V_1\otimes V_2)_\gamma$ with
$\mu'=\gamma^{-1}\circ\mu\circ\gamma$. A basis of
$(V_1\otimes V_2)_\gamma$ is given by choosing a basis of each
component $V_{1\alpha}\otimes V_{2\beta}$ with $\gamma=\beta\alpha$
with $\alpha\in\pi(a,c)$ and $\beta\in\pi(c,b)$, Thus the trace is
non-trivial on the components of $f$ mapping
$V_{1\alpha}\otimes V_{2\beta}\otimes W_\mu$ to
$W_{\mu'}\otimes V_{1\alpha}\otimes V_{2\beta}$. These components
factor as
\[
 V_{1\alpha}\otimes V_{2\beta}\otimes W_\mu
 \xrightarrow{\mathrm{id}\otimes f_2(\beta)}
 V_{1\alpha}\otimes W_{\beta^{-1}\circ\mu\circ\beta}\otimes
 V_{2\beta}\xrightarrow{f_1(\alpha)\otimes\mathrm{id}}
 W_{\mu'}\otimes V_{1\alpha}\otimes V_{2\beta}.
\]
The claim follows by taking the tensor product of bases of $V_{1\alpha}$
and $V_{2\beta}$.
\end{proof}

\begin{Corollary}\label{c-Transfer} Suppose that
 $\check R_{V_iV_j}\in\operatorname{Hom}(V_i\otimes V_j,V_j\otimes
 V_i)$, $(1\leq i<j\leq 3)$ are a solution of the Yang--Baxter
 equation
 \[
 \check R_{V_1V_2}^{(23)}\check R_{V_1V_3}^{(12)}\check
 R_{V_2V_3}^{(23)} = \check R_{V_2V_3}^{(12)}\check
 R_{V_1V_3}^{(23)}\check R_{V_1V_2}^{(12)}
 \]
 with invertible $\check R_{V_1V_2}$.
 Then the transfer matrices
 \[
 T_i=\operatorname{tr}_{V_i}\check R_{V_iV_3}\in
 \Gamma(\pi,\operatorname{\underline{End}}(V_3)), \qquad i=1,2.
 \]
 commute: $T_1T_2=T_2T_1$.
\end{Corollary}
\begin{proof} We can write the Yang--Baxter equation as
 \[
 (\mathrm{id}\otimes\varphi)\check R_{V_2V_3}^{(12)}\check
 R_{V_1V_3}^{(23)}\big(\varphi^{-1}\otimes \mathrm{id}\big) =
 \check R_{V_1V_3}^{(12)}\check R_{V_2V_3}^{(23)}
 \]
 with $\varphi=\check R_{V_1V_2}^{-1}$.
 By Lemma~\ref{l-JB} (i) we deduce that
 \[
 \operatorname{tr}_{V_2\otimes V_1}\check R_{V_2V_3}^{(12)}\check
 R_{V_1V_3}^{(23)}=\operatorname{tr}_{V_1\otimes V_2}
 \check R_{V_1V_3}^{(12)}\check R_{V_2V_3}^{(23)}.
 \]
 The claim then follows from Lemma~\ref{l-JB}(ii).
\end{proof}
\begin{Remark}
 For $W=\mathbf 1$ the $R$-matrix $R_{V,\mathbf 1}$ is the tautological
 map $V\otimes \mathbf 1\to \mathbf 1\otimes V$ of Example~\ref{e-4}.
 Thus the transfer matrix generalizes the notion of character.
\end{Remark}
\begin{Remark} In particular Corollary \ref{c-Transfer} applies to the
 case of the category $M(R,\pi)$: for each object $(W,L_W)$ we can
 interpret the {\it RLL} relations of Section \ref{ss-RLL} as a~Yang--Baxter equation with $V_1=V(z)$, $V_2=V(w)$ (see Example~\ref{e-0}), $V_3=W$ and $\check R_{VW}=L_W(z)$. One gets the
 basic statement of the quantum inverse scattering method that the
 transfer matrices $\operatorname{tr}_VL_W(z)$ commute for different
 values of the spectral parameter $z$. In that language Corollary is
 the generalization of this statement to the case of arbitrary
 ``auxiliary spaces'' $V_1$, $V_2$.
\end{Remark}

\subsection{Action groupoids and dynamical Yang--Baxter equation}
If $\pi=A\rtimes G$ is an action groupoid, $R$-matrices for
$\pi$-graded vector spaces are expressed in terms of the graded
components as dynamical $R$-matrices. Then the tensor product
of $\pi$-graded modules is
\[
 (V\otimes W)_{(a,g)}=\sum_{h\in G} V_{(a,h)}\otimes W_{(ah,h^{-1}g)}.
\]
Let $\check R(z)$ be a solution of
the Yang--Baxter equation \eqref{e-YB} and let
\[
 \check R(z,a)\in\oplus_{g\in G}\operatorname{End}_k((V\otimes
 V)_{(a,g)})
\]
be the restriction of $\check R(z)$ to the graded
components with fixed $a\in A$. Then the Yang--Baxter equation can be
written as
\begin{gather*}
\check R^{(23)}\big(z-w,ah^{(1)}\big)
 \check R^{(12)}(z,a)\check R^{(23)}\big(w,ah^{(1)}\big) = \check R^{(12)}(w,a)\check R^{(23)}\big(z,ah^{(1)}\big)\check R^{(12)}(z-w,a).
\end{gather*}
Here we use the ``dynamical'' notation with the placeholder $h^{(i)}$:
\[
 \check R^{(23)}\big(ah^{(1)}\big)(u\otimes v\otimes w)= u\otimes \check
 R(ag)(v\otimes w) \qquad\text{if} \quad u\in V_{(a,g)}.
\]
If we compose on the left with the product
$p^{(23)}p^{(13)}p^{(12)}=p^{(12)}p^{(13)}p^{(23)}$ of
flips $p\colon v\otimes w\mapsto w\otimes v$ we get the YBE for $R=
p\circ \check R$ in the form
\begin{gather*}
 R^{(12)}\big(z-w,ah^{(3)}\big) R^{(13)}(z,a) R^{(23)}\big(w,ah^{(1)}\big)
=R^{(23)}(w,a)R^{(13)}\big(z,ah^{(2)}\big)R^{(12)}(z-w,a).
\end{gather*}

\subsection{Transfer matrices in the case of action groupoids}\label{ss-3.7}
In the case of action groupoid $\pi=A\rtimes G$ we can identify
convolution algebras with coefficients in $\pi$-graded endomorphisms
with algebras of difference operators (or discrete connections) acting
on the sections of sheaves over~$A$. Let $G(a)=\{g\in G\,|\, ag=a\}$
denote the stabilizer subgroup of an object $a\in A$ and for
$W\in \operatorname{Vect}_k(\pi)$ let
$W_{G(a)}=\oplus_{g\in G(a)}W_{(a,g)}$. Let $\Gamma(A,W)$ be the space
of maps $\psi\colon A\to\sqcup_{a\in A}W_{G(a)}$ such that
$\psi(a)\in W_{G(a)}$ for all $a\in A$. Then $\Gamma(A,W)$ is
naturally a module over $\Gamma(\pi,\operatorname{\underline{End}}W)$
and the transfer matrices can be realized as linear operators on
$\Gamma(A,W)$.

Explicitly, we have
$(\operatorname{\underline{End}}W)_{(a,g)}=\oplus_{h\in G(a)}
\operatorname{Hom}(W_{(ag,g^{-1}hg)},W_{(a,h)})$. For
$\psi\in\Gamma(A,W)$ let $(t_g\psi)(a)=\psi(ag)\in W_{G(ag)}$.
Then $f\in\Gamma(\pi,\operatorname{\underline{End}}W)$ acts on $\psi\in\Gamma(A,W)$
as
\[
 (f\psi)(a)=\sum_{g\in G}f(a,g)(t_g\psi)(a).
\]
Let $\check R_{UW}[a,g]$ be the component
$U_{(a,g)}\otimes W_{G(ag)}\to W_{G(a)}\otimes U_{(a,g)}$ of a dynamical
$R$-matrix $U\otimes W\to W\otimes U$. Then
$r_g(a)=\operatorname{tr}_{U_{(a,g)}}\check R_{UW}[a,g]$ maps $W_{G(ga)}$ to
$W_{G(a)}$. The transfer matrix is
\[
 \operatorname{tr}_U\check R_{UW} =\sum_{g\in G}r_g t_g,
\]
where $r_g$ is understood as a multiplication operator.

\section{Elliptic quantum groups}\label{s-4}
As a class of examples of the above construction let us work out the
case of the dynamical $R$-matrix defining the elliptic quantum group
in the $\mathfrak{gl}_{n}$-case, see \cite{FelderICM1995}.
Here $k=\mathbb C$.
\subsection{Action groupoids of the weight lattice} \label{s-4.1} Let
$\mathfrak h\cong \mathbb C^{n}$ be the Lie subalgebra of diagonal
matrices in $\mathfrak {gl}_{n}$ and $\mathfrak h_0\subset \mathfrak h$
the subalgebra of traceless diagonal matrices. The weight lattice is
the lattice $P=\sum_{i=1}^{n}\mathbb Z\epsilon_i\subset \mathfrak h^*$
spanned by the coordinate functions $\epsilon_i\colon x\mapsto x_i$. It acts
on $\mathfrak h_0^*=\mathfrak h^*/\mathbb C(1,\dots,1)$ by
translations on $\mathfrak h^*$ composed with the canonical projection
$\mathfrak h^*\to \mathfrak h^*_0$. For each orbit $O_b=b+P$,
$b\in\mathfrak h_0^*$ we have a groupoid $O_b\rtimes P$, consisting of
pairs $(a,\mu)\in O_b\times P$ with composition
$(a_1,\mu_1)\circ(a_2,\mu_2)=(a_2,\mu_1+\mu_2)$, defined if
$a_1=a_2+\mu_2$.

We consider the following standard subsets of $\mathfrak h_0^*$:
\begin{itemize}\itemsep=0pt
\item
 The set of $\mathfrak{sl}_n$-dominant weights
 $P_+=\{\lambda\in \mathbb Z^{n}/\mathbb Z(1,\dots,1)
 \,|\, \lambda_1\geq\cdots\geq\lambda_n\}$.
\item
 The set of regular dominant weights
 $P_{++}=\{\lambda\in P_+\,|\, \lambda_1>\cdots>\lambda_n\}$.
\item
 The set of dominant affine weights
 $P_{+}^r=\{\lambda\in P_{+}\,|\, \lambda_1-\lambda_n\leq r\}$, of level
 $r\in\mathbb Z_{\geq 0}$.
\item
 The set of regular dominant affine weights
 $P_{++}^r=\{\lambda\in P_{++}\,|\, \lambda_1-\lambda_n<r\}$, of level
 $r\in\mathbb Z_{\geq n}$.
\end{itemize}
We have a bijection $P_+^r\to P_+^{r+n}$ given by $\lambda\mapsto \lambda +\rho$,
$\rho=(n-1,\dots,1,0)$.

\subsection[Dynamical R-matrix]{Dynamical $\boldsymbol{R}$-matrix}
Fix two complex numbers $\tau$, $\gamma$ such that $\operatorname{Im}\tau>0$
and $\gamma\not\in\mathbb Z+\tau \mathbb Z$. Let
\[
 \theta(z,\tau)=-\sum_{n\in\mathbb Z}
 {\rm e}^{{\rm i}\pi (n+\frac12)^2\tau+2\pi {\rm i} (n+\frac12) (z+\frac12)}
\]
be the odd Jacobi theta function and
$[z]=\theta(\gamma z,\tau)/(\gamma\theta'(0,\tau))$ is normalized to
have derivative~1 at $z=0$. The function $z\mapsto[z]$ of one complex
variable is an odd entire function with first order zeros on the
lattice
\[
 \Lambda=\textstyle \mathbb Z\frac1\gamma+\mathbb Z\frac{\tau}{\gamma}.
\]
The defining representation $\bar V=\mathbb C^{n}$ of $\mathfrak{gl}_{n}$
has a weight decomposition
$\bar V=\oplus_{i=1}^{n} \bar V_{\epsilon_i}$ where
$\bar V_{\epsilon_i}=\mathbb C e_i$ is the span of the $i$-th standard
basis vector.

Let $E_{ij}$ be the ${n}\times {n}$ matrix such that $E_{ij}e_k=\delta_{jk}e_i$
for all $k\in\{1,\dots,{n}\}$.
The (unnormalized) elliptic dynamical $R$-matrix with spectral parameter
$z\in\mathbb C$ is
\begin{gather} \check R(z,a) =\sum_{i=1}^{n}
 E_{ii}\otimes E_{ii}
 -\sum_{i\neq j=1}^{n}\frac{[a_i-a_j+1][z]}{[a_i-a_j][1-z]}
 E_{ij}\otimes E_{ji}\notag
 \\
\hphantom{\check R(z,a) =}{}
 +\sum_{i\neq j=1}^{n}
 \frac{[a_i-a_j+z][1]}{[a_i-a_j][1-z]}E_{ii}\otimes E_{jj}.\label{e-R}
\end{gather}
It is a meromorphic function of $z\in\mathbb C$, $a\in\mathbb C^{n}$ and
solves the dynamical Yang--Baxter equation in the additive form
\begin{gather*}
 \check R\big(z-w,a+h^{(1)}\big)^{(23)} \check R(z,a)^{(12)} \check R\big(w,a+h^{(1)}\big)^{(23)} \\
\qquad{}= \check R(w,a)^{(12)} \check R\big(z,a+h^{(1)}\big)^{(23)} \check R(z-w,a)^{(12)},
\end{gather*}
and the inversion relation
\[
 \check R(z,a)\check R(-z,a) = \mathrm{id}_{\mathbb C^{n}\otimes\mathbb C^{n}},
\]
valid for generic $z$, $a$.
The relation to the $R$ matrix presented in~\cite{FelderICM1995} is
$\check R(z,a)= PR(z,\lambda)$ where $\lambda=\gamma a$.
By restricting $a$ to take values in an orbit~$O_b$
we can construct $R$-matrices acting on groupoid-graded vector
spaces. To do this we need to avoid the poles on the hyperplanes
$a_i-a_j\equiv 0\mod\Lambda$, $(i\neq j)$ of the $R$-matrix.

We consider the case where $r=1/\gamma$ is an integer $> n$.
Then for $a\in P$ the poles are at
\[
 a_i-a_j=m r,\qquad i\neq j\in\{1,\dots,{n}\},\qquad m\in\mathbb Z.
\]
Let $\Delta\subset P$ be the union of these hyperplanes. The affine
Weyl group $W_r$ is the group generated by orthogonal reflections at
these hyperplane in the Euclidean space
$P\otimes \mathbb R=\mathbb R^{n}$. It acts freely and transitively
on the complement of $\Delta\otimes\mathbb R$ and $P_{++}^r$ is the
set of weights in a connected component of the complement and is a
fundamental domain for the action of~$W_r$ on $P\smallsetminus\Delta$.

We distinguish two cases, named after the corresponding models of
statistical mechanics.
\begin{enumerate}\itemsep=0pt
\item{\bf Generalized SOS model.} Let $b\in \mathfrak h^*$, $\pi=O_b\rtimes P$,
 $V^b=\oplus_{\gamma\in\pi}V^b_\gamma$ with
 \[
 V^{b}_{(a,\mu)}=\begin{cases} \bar V_{\epsilon_i}=\mathbb Ce_i
 &\text{if $\mu=\epsilon_i$, $i=1,\dots,{n}$ and $a\in O_b$,}
 \\
 0,&\text{if $\mu\not\in\{\epsilon_1,\dots,\epsilon_n\}$.}
 \end{cases}
 \]
 If $b$ does not lie in $\Delta$ then $R(z)$ is a well-defined endomorphism
 of $V^b\otimes V^b$ and obeys the Yang--Baxter equation.
\item{\bf Generalized RSOS model.} Let $b=0$. The groupoid $\pi$ is the full
 subgroupoid of $O_0\rtimes P$ on $P^r_{++}$: it consists of pairs
 $(a,\mu)\in O_0\rtimes P$ such that both $a$ and $a+\mu$ lie in
 $P^{r}_{++}$. The non-zero components are
 \[
 V^{{\rm RSOS}}_{(a,\epsilon_i)}=\bar V_{\epsilon_i}=\mathbb C{e_i},\qquad
 a,a+\epsilon_i\in P^{r}_{++}.
 \]
\end{enumerate}
To define $\check R(z)$ in the RSOS case we view $V^{{\rm RSOS}}$
as an $O_0\rtimes P$-graded subspace of $V^{b=0}$ such that
$V^{{\rm RSOS}}_\gamma=0$ for $\gamma\not\in \pi$, see Section~\ref{ss-subg}.
\begin{Proposition}
 Let $V=V^{b=0}$ and let $V^{{\rm RSOS}}\subset V$ be viewed as an
 $O_0\rtimes P$-graded subspace of $V$. Then $\check R(z)$ is
 well-defined on $V^{{\rm RSOS}}\otimes V^{{\rm RSOS}}$ and
 maps $V^{{\rm RSOS}}\otimes V^{{\rm RSOS}}$ to itself.
\end{Proposition}
\begin{proof} While $\check R(z)$ is not defined on all vectors in
 $V\otimes V$, it is well-defined on
 $V^{{\rm RSOS}}\otimes V^{{\rm RSOS}}$ since by construction
 the denominators $[a_i-a_j]$ don't vanish for $a\in
 P_{++}^r$. Suppose $(a+\epsilon_j,\epsilon_i)$ and $(a,\epsilon_j)$
 are composable arrows in the subgroupoid $\pi$, meaning that
 $a$, $a+\epsilon_j$ and $a+\epsilon_i+\epsilon_j$ belong to $P_{++}^r$. Then $\check R(z)$
 maps $V_{(a,\epsilon_j)}\otimes V_{(a+\epsilon_j,\epsilon_i)}$ to itself
 if $i=j$ and to
 \[
V_{(a,\epsilon_j)} \otimes V_{(a+\epsilon_j,\epsilon_i)} \oplus
 V_{(a,\epsilon_i)} \otimes V_{(a+\epsilon_i,\epsilon_j)}
 \]
 if $i\neq j$. The first summand is indexed by a pair of arrows in the subgroupoids
 but the second is not if $a+\epsilon_i\not\in P_{++}^r$.
 Thus $\check R(z)$ preserves $V^{{\rm RSOS}}$ if and only if
 the component of the image of~$\check R(z)$ in
 $V_{(a,\epsilon_i)}\otimes V_{(a+\epsilon_i,\epsilon_j)}$
 for $a+\epsilon_i\not\in P^r_{++}$ vanishes.
 So we need to check the vanishing
 of the component
 \begin{equation}\label{e-5}
 -\frac{[a_i-a_j+1][z]}{[a_i-a_j][1-z]}E_{ij}\otimes E_{ji}\colon \
 V_{(a,\epsilon_j)}\otimes V_{(a+\epsilon_j,\epsilon_i)}\to
 V_{(a,\epsilon_i)}\otimes V_{(a+\epsilon_i,\epsilon_j)},\qquad i\neq j,
 \end{equation}
 of $\check R(z,a)$ in the case where $a+\epsilon_i\not\in P^r_{++}$ and
 $a,a+\epsilon_j,a+\epsilon_i+\epsilon_j\in P^{r}_{++}$. The condition for~$a$
 being in $P^r_{++}$ is
 $a_1>\cdots>a_n>a_1-r$. For
 $a\in P_{++}^r$, $a+\epsilon_i$ violates an inequality if and only
 if $i\geq 2$ and $a_{i-1}=a_i+1$ or $i=1$ and $a_n=a_1-r+1$. The
 condition that $a+\epsilon_i+\epsilon_j\in P^{r}_{++}$ implies that
 $j= i-1$ if $i\geq 2$ and $j=n$ if $i=1$. In both cases
 $a_j\equiv a_i+1\operatorname{mod}r$ and thus \eqref{e-5} vanishes.
\end{proof}

\begin{Corollary} The restriction of $\check R(z)$ to
 $V^{{\rm RSOS}}\otimes V^{{\rm RSOS}}$ obeys the dynamical Yang--Baxter
 equation.
\end{Corollary}

\begin{Example}
 The case $n=2$ is the original RSOS model of
 \cite{AndrewsBaxterForrester1984}. In this case we have a bijection
 $O_0=\mathbb Z\times \mathbb Z/(\mathbb Z(1,1))\to \mathbb Z$
 sending the class of $(a_1,a_2)$, to $l=a_1-a_2$ under this
 identification $A=P^{r}_{++}=\{1,\dots,r-1\}$.

The dynamical $R$-matrix in the basis $e_1\otimes e_1$, $e_1\otimes e_2$, $e_2\otimes e_1$, $e_2\otimes e_2$ is
\begin{equation*}
\check{R}(z,a)=
\begin{pmatrix}
 1& 0 & 0 & 0 \\
 0&\dfrac{[l+z][1]}{[l][1-z]} &-\dfrac{[l+1][z]}{[l][1-z]} & 0 \vspace{1mm}\\
 0&-\dfrac{[l-1][z]}{[l][1-z]} &\dfrac{[l-z][1]}{[l][1-z]} & 0 \\
 0&0 &0 & 1
\end{pmatrix}
=
\begin{pmatrix}
 W_1& 0 & 0 & 0 \\
 0 &W_5 &W_4 & 0 \\
 0 &W_3 &W_6 & 0 \\
 0 &0 &0 & W_2
\end{pmatrix}.
\end{equation*}
The Boltzmann weights $W_j$ correspond to the local configurations of Fig.~\ref{drawing}.

\begin{figure}[h]\centering
	\begin{tikzpicture}
		\draw (0,0)--(2,0)--(2,2)--(0,2)--(0,0);
		\node at (0,2) [circle,fill,inner sep=1.5pt]{};
		\node at (2,0) [circle,fill,inner sep=1.5pt]{};
		\draw (0.5,0.3) node{\small $l+1$};
		\draw (1.5,0.3) node{\small $l+2$};
		\draw (1.5,1.7) node{\small $l+1$};
		\draw (0.3,1.7) node{\small $l$};
		\draw (1,-0.5) node{$W_1$};
		\draw (2.5+0,0)--(2.5+2,0)--(2.5+2,2)--(2.5+0,2)--(2.5+0,0);
		\node at (2.5+0,2) [circle,fill,inner sep=1.5pt]{};
		\node at (2.5+2,0) [circle,fill,inner sep=1.5pt]{};
		\draw (2.5+0.5,0.3) node{\small $l-1$};
		\draw (2.5+1.5,0.3) node{\small $l-2$};
		\draw (2.5+1.5,1.7) node{\small $l-1$};
		\draw (2.5+0.3,1.7) node{\small $l$};
		\draw (2.5+1,-0.5) node{$W_2$};
		\draw (5+0,0)--(5+2,0)--(5+2,2)--(5+0,2)--(5+0,0);
		\node at (5+0,2) [circle,fill,inner sep=1.5pt]{};
		\node at (5+2,0) [circle,fill,inner sep=1.5pt]{};
		\draw (5+0.5,0.3) node{\small $l+1$};
		\draw (5+1.7,0.3) node{\small $l$};
		\draw (5+1.5,1.7) node{\small $l-1$};
		\draw (5+0.3,1.7) node{\small $l$};
		\draw (5+1,-0.5) node{$W_3$};
		\draw (7.5+0,0)--(7.5+2,0)--(7.5+2,2)--(7.5+0,2)--(7.5+0,0);
		\node at (7.5+0,2) [circle,fill,inner sep=1.5pt]{};
		\node at (7.5+2,0) [circle,fill,inner sep=1.5pt]{};
		\draw (7.5+0.5,0.3) node{\small $l-1$};
		\draw (7.5+1.7,0.3) node{\small $l$};
		\draw (7.5+1.5,1.7) node{\small $l+1$};
		\draw (7.5+0.3,1.7) node{\small $l$};
		\draw (7.5+1,-0.5) node{$W_4$};

		\draw (10+0,0)--(10+2,0)--(10+2,2)--(10+0,2)--(10+0,0);
		\node at (10+0,2) [circle,fill,inner sep=1.5pt]{};
		\node at (10+2,0) [circle,fill,inner sep=1.5pt]{};
		\draw (10+0.5,0.3) node{\small $l+1$};
		\draw (10+1.7,0.3) node{\small $l$};
		\draw (10+1.5,1.7) node{\small $l+1$};
		\draw (10+0.3,1.7) node{\small $l$};
		\draw (10+1,-0.5) node{$W_5$};

		\draw (12.5+0,0)--(12.5+2,0)--(12.5+2,2)--(12.5+0,2)--(12.5+0,0);
		\node at (12.5+0,2) [circle,fill,inner sep=1.5pt]{};
		\node at (12.5+2,0) [circle,fill,inner sep=1.5pt]{};
		\draw (12.5+0.5,0.3) node{\small $l-1$};
		\draw (12.5+1.7,0.3) node{\small $l$};
		\draw (12.5+1.5,1.7) node{\small $l-1$};
		\draw (12.5+0.3,1.7) node{\small $l$};
		\draw (12.5+1,-0.5) node{$W_6$};
		
	\end{tikzpicture}
	\caption{Six possible types of Boltzmann weight in the eight-vertex SOS model.} \label{drawing}
\end{figure}
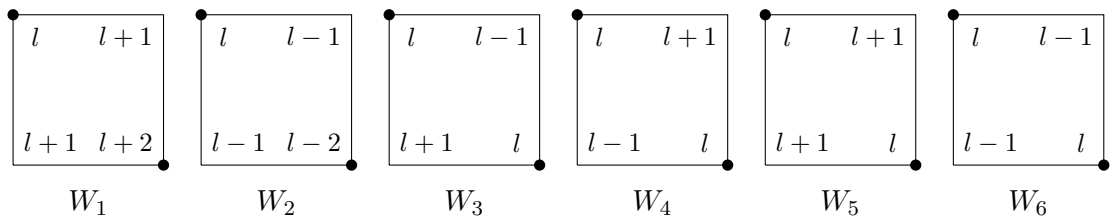

The hard hexagon model is part of the family of RSOS models with
$r=5$. In this case the allowed values of $l$ are $1$, $2$, $3$, $4$ and,
because of the identity $[u]=[5-u]$, $u\in\mathbb C$, the Boltzmann
weights are invariant under $l\mapsto 5-l$. We can map the model to a
lattice gas model on a square lattice. A configuration of the RSOS
model, i.e., a map from the vertices of a square lattice to
$\{1,2,3,4\}$ is mapped to a configuration of particles on the
lattice: there is a particle at the vertices with value~1 or~4 and no
particle at the vertices with value 2 or 3. To each configuration of
particles there correspond two RSOS configuration related by
$l\mapsto 5-l$ and thus with the same Boltzmann weight. The rule that
the value at neighbouring lattice sites differs by~1 translate to the
rule that no two particles of the lattice gas can sit at nearest
neighbouring sites. We can think of this model as a hard square model:
the squares whose vertices are the nearest neighbours of the positions
of the particles are required not to overlap (i.e., to have disjoint
interiors). The row-to-row transfer matrices of these models commute
among themselves and in particular with the transfer matrix with
spectral parameter $z=-1$, which is the transfer matrix of the hard
hexagon model. In this case $W_5$ vanishes for $l=1$ and $W_6$
vanishes for $l=4$, implying that in a configuration with non-zero
Boltzmann weight no two particles can sit at the endpoints of a NW-SE
diagonal. We can translate these rules by thinking of the particles as
midpoints of lattice hexagons which are not allowed to overlap.
 \end{Example}

\subsection{Characters} Here we consider the RSOS case and write $V$
instead of $V^{{\rm RSOS}}$ for the vector representation.

Out of the vector evaluation representation $V(z)$ one can construct
several new representations by the fusion or reproduction method
\cite{KirillovReshetikhin1993, KulishReshetikhinSklyanin1981}, which
admit pairwise $R$-matrices for generic values of the evaluation
parameters. It follows that their characters form a unital commutative
algebra of difference operators with integer coefficients. The unit
element is the character of the trivial representation.

Here are some examples of character calculations. The groupoid
$\pi_r(P)$ of the restricted elliptic quantum group with weight
lattice $P$ is the full subgroupoid of $O_0\rtimes P$ on the set
$A=P^{r}_{++}$. Its convolution ring is isomorphic to the subring
$\chi_AD_P(O_0)\chi_A$ of the crossed product
$D_P(O_0)=\mathbb Z^{O_0}\rtimes \mathbb ZP$, see Lemma~\ref{l-subgr}
and Proposition~\ref{p-difference}. The subring
$\mathbb Z P =\mathbb Z\big[t_1^{\pm1},\dots,t_n^{\pm1}\big]$ is the ring of
Laurent polynomials with generators $t_i=t_{\epsilon_i}$,
$i=1,\dots,n$.
\begin{itemize}\itemsep=0pt
\item The character of the trivial representation is the
 multiplication operator by the characteristic function of
 $A=P_{++}^r$:
 \[
 \operatorname{ch}_{\mathbf 1}=\chi_{A}.
 \]
\item The character of the vector representation $V(z)$.
 \[
 \operatorname{ch}_{V(z)}=\sum_{i=1}^{n} \chi_A t_i
 \chi_A=\sum_{i=1}^n\chi_{A\cap (A-\epsilon_i)}t_i.
 \]
\item The $R$-matrix \eqref{e-R} has a pole at $z=1$ and is not invertible
 at $z=-1$. We set
 $\check R_{{\rm reg}}(1,a)=\operatorname{res}_{z=1}\check R(z,a)$. Then we have
 an exact sequence (see Appendix~\ref{appendixA})
 \[
 \cdots \xrightarrow{\check R(-1)} V(z+1)\otimes V(z)
 \xrightarrow{\check R_{{\rm reg}}(1)}V(z)\otimes V(z+1)
 \xrightarrow{\check R(-1)}\cdots.
 \]
 The analogue of the symmetric square of the vector representation,
 is
 $S^2V(z)=\operatorname{Ker}\check R_{{\rm reg}}(1)$ $\cong
 \operatorname{Coker}\check R(-1)$. Its character can be computed
 from the explicit basis of Lemma \ref{l-A1}
 \[
 \operatorname{ch}_{S^2V(z)}= \sum_{i=1}^n\chi_{A\cap (A-2\epsilon_i)}t_i^2
 +\sum_{1\leq i<j \leq n} \chi_{A\cap (A-\epsilon_i)\cap (A-\epsilon_j)}t_it_j.
 \]
 Similarly, the second exterior power
$\bigwedge^2V(z) =\operatorname{Coker}\check R_{{\rm reg}}(1)\cong \operatorname{Ker}\check R(-1)$
 has character
 \[
 \operatorname{ch}_{\bigwedge^2V(z)}= \sum_{1\leq i<j \leq n} \chi_{A\cap
 (A-\epsilon_i-\epsilon_j)}t_it_j.
 \]
 From the exact sequence $0\to S^2V(z)\to V(z+1)\otimes V(z)\to
 \bigwedge^2V(z)\to 0$, follows the identity
 \[
 \operatorname{ch}_{V(z)}^2= \operatorname{ch}_{\bigwedge^2V(z)}+ \operatorname{ch}_{S^2V(z)}.
 \]
\item For $n=2$ we have the bijection
 $O_0=\mathbb Z\times \mathbb Z/\mathbb Z(1,1)\to \mathbb Z$ sending
 the class of $(a_1,a_2)$ to $a_1-a_2$. Under this identification,
 $A=[1,r-1]=\{1,2,\dots,r-1\}$. Then
 $V(z+p-1)\otimes \cdots \otimes V(z+1)\otimes V(z)$ has a
 subrepresentation
 $S^pV(z)=\cap_{i=1}^{p-1}\operatorname{Ker}\check R_{{\rm reg}}(1)^{(i,i+1)}$
 with character
 \[
 \operatorname{ch}_{S^pV(z)}=\chi_{[1,r-p-1]}t_1^{p}+\chi_{[2,r-p]}t_1^{p-1}t_2+
 \cdots +\chi_{[p+1,r-1]}t_2^{p},
 \]
 $p=0,\dots,r-2$. The characters
 $L_p=\operatorname{ch}_{S^pV(z)}$ obey the fusion rules
 \[
 L_pL_q=\sum_{s\equiv
 p+q\;\mathrm{mod}\;2}N_{pq}^su^{\frac{p+q-s}2}L_s,
 \]
 where $u$ is the central element $t_1t_2$ and
 \[
 N_{pq}^s=\begin{cases}1,&|p-q|\leq s
 \leq\min(p+q,2r-4-p-q),\\
 0,&\text{otherwise.}
 \end{cases}
 \]
 These are the famous fusion rules that first appeared in conformal
 field theory \cite{BelavinPolyakovZamolodchikov1984,
 Verlinde1988}. The algebra with generators $L_p$ and relations
 above is called the Verlinde algebra.
\end{itemize}
\subsection{Exterior powers}
The $k$th exterior power $\bigwedge^kV(z)$ is defined as the quotient
of $V(z+k-1)\otimes\cdots\otimes V(z+1)\otimes V(z)$ by the sum of the
images of $\check R_{{\rm reg}}(1)^{(j,j+1)}$ for $j=1,\dots,k-1$. Its character
is
\[
 \operatorname{ch}_{\bigwedge^kV(z)}= \chi_Ae_k(t_1,\dots,t_n)\chi_A,
\]
with $A=P_{++}^r$. Here
$e_k(t_1,\dots,t_n)=\sum_{1\leq i_1<\dots<i_k\leq n}t_{i_1}\cdots
t_{i_k}$ is the $k$th elementary symmetric polynomial. It follows
from the existence of $R$-matrices for pairs of exterior powers that
these characters commute.

The convolution ring $\mathbb Z[\pi_r(P)]\cong\chi_AD_P(O_0)\chi_A$
acts naturally on functions on $A=P^r_{++}$. The characters can be
simultaneously diagonalized.
\begin{Theorem}
 Let $q={\rm e}^{2\pi {\rm i}/r}$ and pick any $n$-th root $q^{1/n}$ of $q$. For
 each $\lambda\in A=P_{++}^r$ let $\psi_\lambda$ be the function on
 $A$ given by
 \[
 \psi_\lambda(a)=q^{-\frac1n\sum_ia_i\sum_i\lambda_i}\det\big(q^{\lambda_ia_j}\big).
 \]
 Then
 \[
 \operatorname{ch}_{\bigwedge^kV(z)}\psi_\lambda=e_k\big(q^{\bar\lambda_1},
 \dots,q^{\bar\lambda_n}\big)\psi_\lambda,\qquad \bar\lambda_i=\lambda_i-\frac1n\sum_j\lambda_j.
 \]
\end{Theorem}
\begin{proof} Let us first ignore the characteristic functions $\chi_A$ and
 consider $\psi_\lambda$ as a function on the whole weight lattice.
 Then for any symmetric polynomial $P(t)$ in $t_1,\dots,t_n$,
 \[
 P(t)\psi_\lambda=P\big(q^{\bar\lambda_1}, \dots,q^{\bar\lambda_n}\big)\psi_\lambda.
 \]
 The function $\psi_\lambda(a)$ is a skew-symmetric $r$-periodic
 function of $a_1,\dots,a_n$. It thus vanishes on the hyperplanes
 $a_i-a_j\equiv 0\mod r$ and in particular on the walls
 $a_{i+1}-a_{i}=0$, $a_1-a_n=r$ forming the complement of
 $A=P_{++}^r$ in $\bar A=P_+^r$. Thus
 $\chi_A\psi_\lambda=\chi_{\bar A}\psi_\lambda$. The monomials
 $t_I=t_{i_1}\cdots t_{i_k}$ with $i_1<\dots<i_k$ map $P_{++}^r$ to
 $P_+^r$. Thus
 \begin{align*}
 \chi_A e_k(t)\chi_A\psi_\lambda
 &=\chi_Ae_k(t)\chi_{\bar A}\psi_\lambda =\sum_{I\subset\{1,\dots,n\},\, |I|=k}
 \chi_{A\cap{t_I^{-1}\bar A}}e_k(t)\psi_\lambda
 \\
 &=\sum_{I\subset\{1,\dots,n\},\, |I|=k}\chi_{A}e_k(t)\psi_\lambda=e_k(q^{\bar\lambda})\chi_A\psi_\lambda.\tag*{\qed}
 \end{align*}\renewcommand{\qed}{}
\end{proof}

\section{Quantum enveloping algebras at roots of unity}\label{s-5}
We are mainly concerned with dynamical $R$-matrices with non-trivial
dependence on the spectral parameter, but it is instructive
to consider the case of constant dynamical $R$-matrices arising from
the representation theory of semisimple Lie algebras and their quantum
versions. The main simplification in this case is that the category
of finite dimensional modules is braided, namely for each
pair of objects $V$, $W$ there is an isomorphism $\tau_{V,W}\colon V\otimes W
\to W\otimes V$, obeying compatibility conditions with the structure
of a monoidal category,
given by the evaluation of the universal $R$-matrix
composed with the permutation of factors. This property fails
for some pairs of objects in the case of quantum affine Lie algebras or
Yangians, which is the case when the $R$-matrices have a non-trivial
dependence on the spectral parameter.

We focus on the case of quantum groups at root of unity which is a toy
model for restricted models. Strictly speaking the above has to be
corrected in this case and one has to be more careful in the
definition of the category of finite dimensional modules. We consider
the semisimple quotient of the category of tilting modules
\cite[Section 11.3]{ChariPressleyBook1994},
\cite{AndersenParadowski1995,Sawin2006}. It is an abelian monoidal
$\mathbb C$-linear ribbon category $C_q(\mathfrak g)$ depending on a
simple Lie algebra $\mathfrak g$ and a primitive $\ell$-th root of
unity $q$. It has finitely many equivalence classes of simple
objects $L_\lambda$ labeled by dominant weights in a scaled Weyl
alcove $P_+^\ell(\mathfrak g)$, and any object is isomorphic to a direct sum of
simple modules. The alcove $P_+^\ell(\mathfrak g)$ is a finite subset of the cone
$P_+$ of dominant weight, bounded by a~hyperplane as in Section~\ref{s-4.1}.
See~\cite{Sawin2006} for a description in the most general case of a~Lie
algebra and any root of unity.

Let $P$ be the weight lattice and $\pi$ be the full subgroupoid of
$P\rtimes P$ on $P_+^\ell(\mathfrak g)$. Its objects are dominant weights and there
is exactly one arrow $a\to b$ for any two weights $a,b\in P_+^\ell(\mathfrak g)$.
As usual we denote by $(a,b-a)$ this arrow.

\begin{Theorem}
 There is a faithful exact monoidal functor
 $U_q\mathfrak g$-$\mathrm{mod}\to\operatorname{Vect}_k(\pi)$ sending
 an object $W$ to $\hat W=\oplus_{(a,\lambda)\in\pi}\hat W_{a,\lambda}$ with
 \[
 \hat W_{a,\lambda} = \operatorname{Hom}_{C_q(\mathfrak g)}
 (L_a,W\otimes L_{a+\lambda}).
 \]
\end{Theorem}
\begin{proof}
 A morphism $f\colon W\to W'$ in $U_q\mathfrak g$-mod induces a
 morphism
 $\hat f\colon\varphi\mapsto (f\otimes\mathrm{id})\circ\varphi$ from~$\hat W$ to~$\hat W'$ and this assignment is compatible with
 compositions so that we get a well-defined functor.
 The trivial module~$k$ which is the
 tensor unit in $U_q\mathfrak g$ is mapped to the tensor unit
 $\hat k=\mathbf 1$ in $\operatorname{Vect}_k(\pi)$, see Section~\ref{ss-tp}.
 Moreover we have a natural transformation
 \[
 \hat W\otimes \hat Z\to \widehat {W\otimes Z},
 \]
 whose restriction to $\hat W_{a,\mu}\otimes \hat Z_{a+\mu,\lambda-\mu}$
 is the composition
 \begin{align*}
 \operatorname{Hom}(L_a,W\otimes L_{a+\mu})\otimes
 \operatorname{Hom}(L_{a+\mu},Z\otimes L_{a+\lambda})
 &\to
 \operatorname{Hom}(L_a,W\otimes Z\otimes L_{a+\lambda}),\\
 \varphi\otimes\psi&\mapsto (\mathrm{id}\otimes\psi)\circ\varphi.\tag*{\qed}
 \end{align*}\renewcommand{\qed}{}
\end{proof}

 Since $C_q(\mathfrak g)$ is semisimple, by taking the direct sum
 over $\mu$ we get an isomorphism
 $(\hat W\otimes \hat Z)_{a,\lambda}\to (\widehat{W\otimes
 Z})_{a,\lambda}$ on each graded component.

\begin{Theorem} There is a system of isomorphisms $($dynamical $R$-matrices$)$
 \[
 \check R_{W,Z}\colon \ \hat W\otimes \hat Z\to \hat Z\otimes \hat W
 \]
 obeying the quasitriangularity relations
 \[
 \check R_{W\otimes W',Z}=\check R_{W,Z}^{(12)}\check
 R_{W',Z}^{(23)},\qquad \check R_{W,Z\otimes Z'}=\check
 R_{W,Z'}^{(23)}\check R_{W,Z}^{(12)}
 \]
 and the dynamical Yang--Baxter equation
 \[
 \check R_{W,Z}^{(12)}\check R_{Y,Z}^{(23)}\check R_{Y,W}^{(12)}=
 \check R_{Y,W}^{(23)}\check R_{Y,Z}^{(12)}\check R_{W,Z}^{(23)},
\]
for any objects $W$, $W'$, $Z$, $Z'$, $Y$.
\end{Theorem}

Let us compute the character of $\hat W$.
Let
$N_{ab}^c=\operatorname{dim}\operatorname{Hom}(L_c,L_a\otimes L_b)$ be the
multiplicity of
$L_c$ in the decomposition of $L_a\otimes L_b$ as a direct sum of simple
objects. The numbers $N_{ab}^c$ are
called fusion coefficients. They are the structure constants of the
fusion ring (or Verlinde algebra) with generators~$e_a$ and product
$e_ae_b=\sum_c N_{ab}^ce_c$.
Then
\[
 \operatorname{ch}_{L_\mu}=\sum_\lambda n_{\lambda,\mu} t_{\lambda},
\]
where $n_{\lambda,\mu}(a)=N_{a,\mu}^{a+\lambda}$.

The commutativity of the characters is the associativity of the
Verlinde algebra.

\appendix
\section[R-matrix at special values]{$\boldsymbol{R}$-matrix at special values}\label{appendixA}

The $R$-matrix \eqref{e-R} has a pole at $z=1\mod \Lambda$. Because of the inversion relation we see that it is regular and invertible except for $z\in\pm1+\Lambda$.
For $z=-1$ we have
\[
 \check R(-1,a)=-\sum_{i=1}^{n}E_{ii}\otimes E_{ii}
 +\sum_{i\neq j}\frac{[a_i-a_j+1][1]}{[a_i-a_j][2]}(E_{ij}\otimes E_{ji}+
 E_{jj}\otimes E_{ii}).
\]
Let $\check R_{{\rm reg}}(1,a)=\mathrm{res}_{z=1}\check R(z,a)$. Then
\[
 \check R_{{\rm reg}}(1,a)=
 \sum_{i\neq j}\frac{[a_i-a_j+1][1]}{[a_i-a_j]}(E_{ij}\otimes E_{ji}-
 E_{ii}\otimes E_{jj}).
\]
\begin{Lemma}\label{l-A1} Let $e_i(a)$ denote the standard basis vector
 $e_i\in\mathbb C^{n}$ in $V_{(a,\epsilon_i)}$.
 \begin{enumerate}\itemsep=0pt
 \item[$(i)$] The image of $\check R(-1,a)$ coincides with the kernel of
 $\check R_{{\rm reg}}(1,a)$ and is spanned by the linearly independent vectors
 \[
 e_i(a)\otimes e_j(a+\epsilon_i)+e_j(a)\otimes e_i(a+\epsilon_j),
 \]
 where $a$, $i$, $j$ are such that $i\leq j$ and
 $a,a+\epsilon_i,a+\epsilon_j,a+\epsilon_i+\epsilon_j\in P_{++}^r$.
 \item[$(ii)$] The image of $\check R_{{\rm reg}}(1,a)$ coincides with the kernel of
 $\check R(-1,a)$ and is spanned by the following
 linear independent vectors:
 \[
 e_i(a)\otimes e_j(a+\epsilon_i)-e_j(a)\otimes e_i(a+\epsilon_j),
 \]
 where $a$, $i$, $j$ are such that $i<j$ and
 $a,a+\epsilon_i,a+\epsilon_j,a+\epsilon_i+\epsilon_j\in P_{++}^r$;
 \[
 e_{i-1}(a)\otimes e_i(a+\epsilon_{i-1}),
 \]
 where $i=2,\dots,{n}$ and $a_{i-1}=a_i+1$;
 \[
 e_{n}(a)\otimes e_1(a+\epsilon_{n}),
 \]
 where $a_{n}=a_1-r+1$.
 \end{enumerate}
\end{Lemma}
\begin{proof} We consider these operators on the weight spaces
 $(V\otimes V)_{(a,\epsilon_i+\epsilon_j)}$, which are non-vanishing
 for $a,a+\epsilon_i+\epsilon_j\in P_{++}^r$.
 Recall that $a\in P_{++}^r$ means $a_1>\cdots>a_{n}>a_1-r$. Thus the
 only case where the numerator $[a_i-a_j+1]$ vanishes is when
 $j=i-1$ and $a_{i-1}=a_i+1$ for $i=1,\dots,{n}$, where we set
 $a_0=a_{n}+r$. In this case $a+\epsilon_i\not\in P_{++}^r$.
 There are three cases to consider (we may assume that $i\geq j$):
 \begin{enumerate}\itemsep=0pt
 \item[(a)] $i=j$. In this case the weight space is spanned by
 $e_i(a)\otimes e_i(a+\epsilon_i)$, $\check R(-1)$ acts by
 multiplication by $-[2]/[1]\neq 0$ and $\check R_{{\rm reg}}(1)$ vanishes. Thus
 \begin{gather*}
 \operatorname{Im}\check R(-1)
 =\operatorname{Ker}\check R_{{\rm reg}}(1) =
 \operatorname{span}( e_i(a)\otimes e_i(a+\epsilon_i)),
 \\
 \operatorname{Im}\check R_{{\rm reg}}(1)
 =\operatorname{Ker}\check R(-1) =0.
 \end{gather*}
 \item[(b)] $i>j$ and $a+\epsilon_i,a+\epsilon_j\in P_{++}^r$. A
 basis of the weight space consists of
 $e_j(a)\otimes e_i(a+\epsilon_j)$ and
 $e_i(a)\otimes e_j(a+\epsilon_i)$. The matrices of $\check R(-1)$
 and $\check R_{{\rm reg}}(1)$ in this basis are (up to factors of $[1]$ and $[2]$)
 \[
 \left(
 \begin{matrix}
 \dfrac{[a_i-a_j+1]}{[a_i-a_j]}
 &
 \dfrac{[a_i-a_j-1]}{[a_i-a_j]}
 \vspace{1mm}\\
 \dfrac{[a_i-a_j+1]}{[a_i-a_j]}
 &
 \dfrac{[a_i-a_j-1]}{[a_i-a_j]}
 \end{matrix}
 \right)
 \qquad \text{and}
 \qquad
 \left(
 \begin{matrix}
 \dfrac{[a_i-a_j-1]}{[a_i-a_j]}
 &
 - \dfrac{[a_i-a_j+1]}{[a_i-a_j]}
 \vspace{1mm}\\
 -\dfrac{[a_i-a_j-1]}{[a_i-a_j]}
 &
 \dfrac{[a_i-a_j+1]}{[a_i-a_j]}
 \end{matrix}
 \right),
 \]
 respectively. The matrix entries are all non-zero and we
 see that
 \begin{gather*}
 \operatorname{Im}\check R(-1)=\operatorname{Ker}\check R_{{\rm reg}}(1) =
 \operatorname{span}( e_j(a)\otimes e_i(a+\epsilon_j)+
 e_i(a)\otimes e_j(a+\epsilon_i)),\\
 \operatorname{Im}\check R_{{\rm reg}}(1)=\operatorname{Ker}\check R(-1) =
 \operatorname{span}(e_j(a)\otimes e_i(a+\epsilon_j)-
 e_i(a)\otimes e_j(a+\epsilon_i)).
 \end{gather*}
 \item[(c)] $i>j$ and one of $a+\epsilon_i,
a+\epsilon_j\not\in P_{++}^r$. This happens only if $j=i-1$ and
 $a_{i-1}=a_i+1$ (with $a_0:=a_{n}+r$). Then the weight space is
 spanned by $e_{i-1}(a)\otimes e_i(a+\epsilon_{i-1})$;
 $\check R(-1)$ acts by zero and $\check R_{{\rm reg}}(1)$ acts by
 multiplication by $[2]/[1]\neq 0$. Thus
 \begin{gather*}
 \operatorname{Im}\check R(-1)
 =\operatorname{Ker}\check R_{{\rm reg}}(1)
 =0,
 \\
 \operatorname{Im}\check R_{{\rm reg}}(1)
 =\operatorname{Ker}\check R_{{\rm reg}}(-1)
 = \operatorname{span} (e_{i-1}(a)\otimes e_i(a+\epsilon_{i-1})),\qquad
 i=1,\dots,{n},
 \end{gather*}
 where for $i=1$ the right-hand side is
 $e_{n}(a)\otimes e_1(a+\epsilon_{n})$.\hfill \qed
 \end{enumerate}\renewcommand{\qed}{}
 \end{proof}

\subsection*{Acknowlegments}
 {\sloppy The authors are supported in part by the National Centre of Competence
in Research SwissMAP~-- The Mathematics of Physics~-- of the Swiss
National Science Foundation. They are also supported by the
grants 196892 and 178794 of the Swiss National Science Foundation, respectively.
We are grateful to the referees for their careful reading of the first version of
this paper and their many useful suggestions and corrections.

}

\pdfbookmark[1]{References}{ref}
\LastPageEnding

\end{document}